\documentclass[12pt,a4paper]{amsart}
\usepackage{amscd}
\usepackage{amsmath}
\usepackage{latexsym}
\usepackage{amsfonts}
\usepackage{amssymb}
\usepackage{amsthm}
\usepackage{graphicx}
\usepackage{verbatim}
\usepackage{mathrsfs}
\usepackage{enumerate}
\usepackage{color}

\newtheorem{definition}{Definition}[section]
\newtheorem{theorem}{Theorem}[section]
\newtheorem{lemma}[theorem]{Lemma}

\newtheorem{proposition}[theorem]{Proposition}

\theoremstyle{remark}
\newtheorem{remark}[theorem]{Remark}

\newcommand{\RR}{\mathbb{R}}

\allowdisplaybreaks
\begin{document}
\title[well-posedness of the  Oldroyd-B model without damping mechanism ]{Global well-posedness in the  critical Besov spaces for the
incompressible Oldroyd-B model without damping mechanism }

\author[]{Qionglei Chen}
\address[]{Institute of Applied Physics and Computational Mathematics, Beijing 100088, China}
\email{chen\_qionglei@iapcm.ac.cn}

\author[]{Xiaonan Hao}
\address[]{The Graduate School of China Academy of Engineering Physics, Beijing 100088, China}
\email{xn\_hao@163.com}

\date{\today}
\subjclass[2000]{35Q30, 76D05; 35B40.} \keywords{incompressible
Oldroyd-B model, critical spaces, global solution.}

\begin{abstract}

We prove the global well-posedness in the critical Besov spaces for
the incompressible Oldroyd-B model without damping mechanism on the
stress tensor in $\mathbb{R}^d$ for the small initial data. Our
proof is based on the observation that the  behaviors of Green's
matrix to the system of
$\big(u,(-\Delta)^{-\frac12}\mathbb{P}\nabla\cdot\tau\big)$  as well
as the effects of $\tau$ change from the low frequencies to the high
frequencies and the construction of the appropriate energies in
different frequencies.
\end{abstract}

\maketitle
\section{Introduction}\label{INTR}
\setcounter{section}{1}\setcounter{equation}{0}

We are concerned with the  incompressible Oldroyd-B model  of the non-Newtonian fluid
in $\mathbb{R}^+\times\mathbb{R}^d$
\begin{equation}\label{OBA}
\begin{cases}
u_t+u\cdot\nabla u-\nu\Delta u+\nabla p=\mu_1\nabla\cdot\tau,\\
\tau_t+u\cdot\nabla\tau+a\tau+Q(\tau,\nabla u)=\mu_2D(u),  \\
\nabla\cdot u=0,\\
u(0,x)=u_0(x),\ \ \tau(0,x)=\tau_0(x).
\end{cases}
\end{equation}
Here $u(t,x)$ stands for the velocity fluid and $\tau(t,x)$ is the
non-Newtonian part of stress tensor ($\tau$ is a $d\times d$
symmetric matrix and $[\nabla\cdot\tau]^i=\sum_j\partial_j\tau^{i,j}$). The pressure $p$ is a scalar and coefficients
$\nu,a,\mu_1,\mu_2$ are assumed to be non-negative constants. The
bilinear term $Q$ has the following form
$$Q(\tau,\nabla u)=\tau W(u)-W(u)\tau-b\big(D(u)\tau+\tau D(u)\big),$$
where $b\in[-1,1]$, and $D(u)=\frac{1}{2}\big(\nabla u+(\nabla
u)^\top\big)$, $W(u)=\frac{1}{2}\big(\nabla u-(\nabla u)^\top\big)$ are the
deformation tensor and the vorticity tensor, respectively. If $a=0$,
we call the system \eqref{OBA} the Oldroyd-B model without damping
mechanism which we investigate in this paper. The Oldroyd-B model is a typical prototypical model for
viscoelastic fluids, which describes the motion of some viscoelastic
flows. For more detailed physical background and derivations about this
model, one refers to \cite{R.B.Bird,J.-Y.Chemin,J.G. Oldroyd}.

The well-posedness of the system \eqref{OBA} had been studied
extensively. In the the case of $a>0$, Guillop\'{e} and Saut
\cite{C.G1} proved that the strong solutions are local well-posed in the
Sobolev space $H^s$. They \cite{C.G2} also showed that these
solutions are global if the coupling parameter and the initial data
are small enough. Their results were extended to the $L^p$
framework by Fernandez-Cara, Guill\'{e}n and Ortega
\cite{E.Fernandez}. Chemin and Masmoudi \cite{J.-Y.Chemin} initiated
the study of the global existence and uniqueness in the critical
Besov spaces, and their results were improved later by Zi, Fang and Zhang
\cite{R.Z.Zi} to the case of the non-small coupling parameter. For
more  results on the well-posedness and the blow-up criterion, one
refers to \cite{J.-Y.Chemin,Miao,R.Z.Zi,D.Y.Fang,Z.Lei} and
references therein.

Now let us say a few words about the so-called critical spaces, for
the incompressible Navier-Stokes equations
\begin{equation*}(\textrm{INS})\qquad
\begin{cases}
u_t+u\cdot\nabla u-\nu\Delta u+\nabla p=0,\\
\nabla\cdot u=0,\\
u(0,x)=u_0(x),
\end{cases}
\end{equation*}
if $(u(t, x), p(t, x))$ is a solution of (INS), then for $\lambda>0$
\begin{equation}\label{scaling of INS}\big(u_\lambda(t, x),p_\lambda(t, x)\big)\triangleq\big(\lambda u(\lambda^2t,
\lambda x), \lambda^2 p(\lambda^2t,
\lambda x)\big),
\end{equation}
is also a solution of (INS). Moreover, the functional space $X$ is
called critical to the system (INS) if the corresponding norm is
invariant under the scaling \eqref{scaling of INS}. It is Obvious
that $\dot{H}^{\frac d2-1}$ is a critical space. Fujita and Kato
\cite{Fujita-Kato} proved the wellposedness of (INS) in
$\dot{H}^{\frac d2-1}$, see also \cite{Can1,Can2,Can3,Can4} and
references therein for the other critical spaces. Although the
system \eqref{OBA} does not have any scaling invariance, one may
find that if the coupling term $\nabla\cdot\tau$ as well as the
damping term $\tau$ is neglected, and $(u,\tau)$ is a solution of
\eqref{OBA}, then for $\lambda>0$
\begin{equation}\label{scaling}
\big(u_\lambda(t,x),\tau_\lambda(t,x),p_\lambda(t,x)\big)\triangleq\big(\lambda
u(\lambda^2 t,\lambda x),\tau(\lambda^2 t,\lambda x),
\lambda^2p(\lambda^2 t,\lambda x)\big)
\end{equation}
is also a solution of \eqref{OBA}. This leads us to define the
following critical spaces of the system \eqref{OBA} as in
\cite{J.-Y.Chemin,R.Z.Zi},
\begin{definition}\label{criticalspace}
A functional space is called critical to \eqref{OBA} if the associated norm is invariant
under the transformation $(u,\tau)\rightarrow(u_\lambda, \tau_\lambda)$
(up to a constant independent of $\lambda$).
\end{definition}
\noindent It is obvious that $\dot{H}^{\frac d2}\times\dot{H}^{\frac
d2-1}$ and $\dot{B}^{\frac dp}_{p,q}\times\dot{B}^{\frac
dp-1}_{p,q}$ are critical spaces. And the reason why one has to
consider the well-posedness in such spaces has been fully explained
in \cite{J.-Y.Chemin}.

As for the researches of other special cases of the system \eqref{OBA}, we sketch some results here. If
$b=0$, the existence of the global weak solution had been proved by
Lions and Masmoudi \cite{P.-L.Lions}. If $\mu=0$ and the
equation of $\tau$ contains viscous term $-\Delta \tau$, the global
existence of small smooth solutions had been proved by Elgindi,
Rousset and Liu \cite{T.M.F,T.M.J}, and the similar result with
general data had also been showed for more special models in
\cite{T.M.F}. When $a=0$, Zhu Yi \cite{zhuyi} constructed
a global smooth solution in $3D$ very recently. More precisely,
\begin{theorem}\cite{zhuyi}\label{zhuyi}
Let $\nu,\mu_1,\mu_2>0$ and $a=0$. Suppose that $\nabla\cdot u=0,\
(\tau_0)^{i,j}=(\tau_0)^{j,i}$ and initial data $\Lambda^{-1}u_0,\
\Lambda^{-1}\tau_0\in H^3(\mathbb{R}^3)$. Then there exists a small
constant $\epsilon$ such that system \eqref{OBA} admits a unique
global classical solution provided that
$$\|\Lambda^{-1}u_0\|_{H^3}+\|\Lambda^{-1}\tau_0\|_{H^3}\leq\epsilon,$$
where $\Lambda^{-1}=(-\Delta)^{\frac{1}{2}}.$
\end{theorem}
It is noted that the required regularity of Theorem 1.1 is far from
the regularity prescribed by the scaling \eqref{scaling}, which
inspires us to consider the well-posedness of \eqref{OBA} ($a=0$) in critical Besov
spaces just like Chemin and Masmoudi had done in \cite{J.-Y.Chemin,R.Z.Zi}
for $a>0$.

Now we state the main result.
\begin{theorem}\label{Maintheorem}
Let $\nu,\mu_1,\mu_2>0$ and  $a=0$. There exists a small constant
$\varepsilon$ such that if
$\tau_0\in\dot{B}^{\frac{d}{2}-1}_{2,1}\cap \dot{B}^{\frac{d}{2}}_{2,1}$,
$u_0\in\dot{B}^{\frac{d}{2}-1}_{2,1}$ with
$$\|u_0\|_{\dot{B}^{\frac{d}{2}-1}_{2,1}}+
\|\tau_0\|_{\dot{B}^{\frac{d}{2}-1}_{2,1}\cap\dot{B}^{\frac{d}{2}}_{2,1}}\leq\varepsilon,$$
then the system \eqref{OBA} has a unique global solution
$(u,\tau)$ such that \begin{eqnarray*}&&u\in
C(\RR^+;\dot{B}^{\frac{d}{2}-1}_{2,1})\cap
L^1(\RR^+;\dot{B}^{\frac{d}{2}+1}_{2,1});\\
&&\tau\in C(\RR^+;\dot{B}^{\frac{d}{2}-1}_{2,1}\cap\dot{B}^{\frac{d}{2}}_{2,1}),\ \ \
\mathbb{P}\nabla\cdot\tau\in
L^1(\RR^+;\dot{B}^{\frac{d}{2}-1}_{2,1}+ \dot{B}^{\frac{d}{2}}_{2,1}). \end{eqnarray*} Here
$\mathbb{P}$ is the projection operator, and $\dot{B}^{s}_{2,1}$ is Besov space. One refers to Section \ref{Besovspace} for its definition.
\end{theorem}
\begin{remark}
When $d=3$, noting that  $H^2\hookrightarrow\dot{B}^{\frac{1}{2}}_{2,1}$ (see Proposition \ref{Besovproperties}),
Theorem \ref{Maintheorem} allows us to involve  a class of functions of the initial data
that Theorem \ref{zhuyi} does not contain. For example,
$$\chi_p(D)|x|^{-\sigma},\ \ \ \ 1<\sigma\le{3}/{2},$$
where $\chi_p(D)f\triangleq\mathcal{F}^{-1}(\chi_{\mathcal{B}(0,p)}\hat{f})$ with the radial function $\chi_{\mathcal{B}(0,p)}\in\mathcal{S}(\mathbb{R}^3)$ supported in the ball $\mathcal{\mathcal{B}}=\{\xi\in\mathbb{R}^d,\ |\xi|\leq p\}$. It is not difficult to check that
$$\chi_p(D)|x|^{-\sigma}\in{\dot{B}^{\frac{1}{2}}_{2,1}(\mathbb{R}^3)};\ \ \ \chi_p(D)|x|^{-\sigma}\notin H^2(\mathbb{R}^3).$$
For the detailed proof, one please refers to
Proposition \ref{example}.
\end{remark}

In \cite{zhuyi}, the author observed that
$(u,\mathbb{P}\nabla\cdot\tau)$ satisfies some kind of the damped
wave equations and  has enough decay regardless of $a=0$, then he
proved Theorem 1.1 by constructing two special time-weighted
energies. However, the approach used in \cite{zhuyi} seems not to
work for the critical spaces. On the other hand, the method used in
\cite{J.-Y.Chemin} for the critical Besov spaces relies heavily on
the damping effect ($a>0$). This makes us to dig out more
informations about the system \eqref{OBA} when $a=0$. We find out
that one part of the stress tensor $\tau$  has damping effect while
the stress tensor itself does not have. More precisely, motivated by
the work of Danchin \cite{R.DANCHIN2000} on the compressible
Navier-Stokes equations, we study the following mixed linear system
\begin{equation}\label{OBLG}
\begin{cases}
u_t-\nu\Delta u-\mu_1\Lambda(\Lambda^{-1}\mathbb{P}\nabla\cdot\tau)=\mathbb{P}E,\ \ \\
(\Lambda^{-1}\mathbb{P}\nabla\cdot\tau)_t+\frac{\mu_2}{2}\Lambda u=\Lambda^{-1}\mathbb{P}\nabla\cdot F, \\
\end{cases}
\end{equation}
where $\mathbb{P}$ is the projection operator,
\begin{equation*}\Lambda=(-\Delta)^{\frac12} \quad\mbox{and}\quad \Lambda^{-1}=(-\Delta)^{-\frac12}.
\end{equation*}
Let $\mathcal{G}(x,t)$ be the Green matrix of system \eqref{OBLG}, we derive
\begin{equation}\label{G}
\hat{\mathcal{G}}(\xi,t)=\left(
\begin{matrix}
\frac{\lambda_+e^{\lambda_+t}-\lambda_-e^
{\lambda_-t}}{\lambda_+-\lambda_-}&
\mu_1|\xi|\frac{e^{\lambda_+t}-e^{\lambda_-t}}{\lambda_+-\lambda_-}\\
-\frac{\mu_2}{2}|\xi|\frac{e^{\lambda_+t}-e^
{\lambda_-t}}{\lambda_+-\lambda_-}&
\frac{\lambda_+e^{\lambda_-t}-\lambda_-e^{\lambda_+t}}{\lambda_+-\lambda_-}
\end{matrix}
\right),
\end{equation}
where
\begin{eqnarray*}&&\lambda_+=\frac{-\nu|\xi|^2+\sqrt{\nu^2|\xi|^4-2\mu_1\mu_2|\xi|^2}}{2},\\
&&\lambda_-=\frac{-\nu|\xi|^2-\sqrt{\nu^2|\xi|^4-2\mu_1\mu_2|\xi|^2}}{2}.\end{eqnarray*}
By analyzing the behaviors of  $\widehat{\mathcal{G}}(x,t)$ in
different frequencies, we discover that $u$ as well as the low
frequencies of $\Lambda^{-1}\mathbb{P}\nabla\cdot\tau$ has a
parabolic smoothing effect, and the high frequencies of
$\Lambda^{-1}\mathbb{P}\nabla\cdot\tau$ have a damping effect.
Specifically, we set the energy in the low frequencies
\begin{equation}\label{elr}
\begin{split}
\mathcal{E}^l_r(t)&\triangleq\sup_t\|u\|^l_{\dot{B}^{\frac{d}{2}-1}_{2,1}}\!\!+\!
\sup_t\|\Lambda^{-1}\mathbb{P}\nabla\!\cdot\!\tau\|^l_{\dot{B}^{\frac{d}{2}-1}_{2,1}}\!+\!
\int_0^t\!\!\|u\|^l_{\dot{B}^{\frac{d}{2}+1}_{2,1}}\!\mathrm{d}t'\!+\!\int^t_0\!\!\|\Lambda^{-1}\mathbb{P}\nabla\!\cdot\!\tau\|^l_{\dot{B}^{\frac{d}{2}+1}_{2,1}}
\!\mathrm{d}t',
\end{split}
\end{equation}and the energy in the high frequencies
\begin{equation}\label{ehr}
\begin{split}
\mathcal{E}^h_r(t)&\triangleq\sup_t\|u\|^h_{\dot{B}^{\frac{d}{2}-1}_{2,1}}+
\sup_t\|\mathbb{P}\nabla\cdot\tau\|^h_{\dot{B}^{\frac{d}{2}-1}_{2,1}}
+\int_0^t\|u\|^h_{\dot{B}^{\frac{d}{2}+1}_{2,1}}\mathrm{d}t'
+\int^t_0\|\mathbb{P}\nabla\cdot\tau\|^h_{\dot{B}^{\frac{d}{2}-1}_{2,1}}\mathrm{d}t',
\end{split}
\end{equation}
where  $\|\ .\ \|^l$ and $\|\ .\ \|^h$ denotes the low and high part
of corresponding norm (see Subsection 2.3). Unfortunately, there
seems no damping effect on another part of the tensor $\tau$, and
the following estimates do not hold
$$\|\tau\|^l_{\dot{B}^{\frac{d}{2}-1}_{2,1}}\le C\|\Lambda^{-1}\mathbb{P}\nabla\!\cdot\!\tau\|^l_{\dot{B}^{\frac{d}{2}-1}_{2,1}}
,\ \ \|\tau\|^l_{\dot{B}^{\frac{d}{2}-1}_{2,1}}\le
C\|\mathbb{P}\nabla\cdot\tau\|^h_{\dot{B}^{\frac{d}{2}-1}_{2,1}}.$$
 Thus it is difficult to deal with some parts of the
nonlinear terms \big(e.g. $Q(\tau,\nabla u)$\big) by the original
energies $\mathcal{E}^l_r(t), \mathcal{E}^h_r(t)$. Naturally, we
want to get more estimates containing term $\tau$ instead of
$\mathbb{P}\nabla\cdot\tau$, and this motivates us to construct the
energies $\mathcal{E}^l(t), \mathcal{E}^h(t)$ as the following,
\begin{equation}\label{el}
\mathcal{E}^l(t)\triangleq\sup_t\|u\|^l_{\dot{B}^{\frac{d}{2}-1}_{2,1}}+
\sup_t\|\tau\|^l_{\dot{B}^{\frac{d}{2}-1}_{2,1}}+\int_0^t\|u\|^l_{\dot{B}^{\frac{d}{2}+1}_{2,1}
}\mathrm{d}t'+\int^t_0\|\Lambda^{-1}\mathbb{P}\nabla\cdot
\tau\|^l_{\dot{B}^{\frac{d}{2}+1}_{2,1}}\mathrm{d}t',
\end{equation}and
\begin{equation}\label{eh}
\mathcal{E}^h(t)\triangleq\sup_t\|u\|^h_{\dot{B}^{\frac{d}{2}-1}_{2,1}}+
\sup_t\|\tau\|^h_{\dot{B}^{\frac{d}{2}}_{2,1}}
+\int_0^t\|u\|^h_{\dot{B}^{\frac{d}{2}+1}_{2,1}}\mathrm{d}t'
+\int^t_0\|\mathbb{P}\nabla\cdot
\tau\|^h_{\dot{B}^{\frac{d}{2}-1}_{2,1}}\mathrm{d}t'.
\end{equation}
Then  the above-mentioned energies  supply $L^\infty_T$ estimate of $\|\tau\|_{\dot{B}^{\frac{d}{2}-1}_{2,1}}$\big($\|\tau\|_{\dot{B}^{\frac{d}{2}}_{2,1}}$\big) in the low (high) frequencies.
For more details, please see Section \ref{PRIORI}. Let us emphasize that if we handle the nonlinear terms of $\tau$ involved, we shall appeal to the way used in \cite{zhuyi}, i.e.,
$$\mathbb{P}\nabla\cdot\tau=\mathbb{P}(u\cdot\nabla\mathbb{P}\nabla\cdot\tau)+\text{some terms containing}\  \nabla u.$$

\vspace{.2cm}

\noindent\textbf{Notations.}
Throughout this paper, $C$ stands for the constant and changes from line to line.
We use $\widehat{u}$ and $\mathcal{F}(u)$ to denote the Fourier transform of $u$.


\section{Littlewood-Paley theory and Besov spaces}\label{Besovspace}
\setcounter{section}{2}\setcounter{equation}{0}

\subsection{Littlewood-Paley decomposition}
Now we introduce the Littlewood-Paley decomposition, which relies on
the dyadic partition of unity, and we can refer to \cite{H.BAHOURI}
for more details. Let us choose two radial functions $\varphi,
\chi\in\mathcal{S}(\mathbb{R}^d)$ supported in
$\mathcal{C}=\{\xi\in\mathbb{R}^d,\
\frac{3}{4}\leq|\xi|\leq\frac{8}{3}\}$ and
$\mathcal{B}=\{\xi\in\mathbb{R}^d,\ |\xi|\leq\frac{4}{3}\}$
respectively such that
$$\sum_{j\in\mathbb{Z}}\varphi(2^{-j}\xi)=1\ \ \text{if}\ \ \xi\neq0.$$
Denote $h(x)=\mathcal{F}^{-1}\big(\varphi(\xi)\big)$, we define the
dyadic blocks as follows
\begin{eqnarray*}
&&\dot{\Delta}_j u=\varphi(2^{-j}D)u=2^{jd}\int_{\mathbb{R}^d}h(2^jy)u(x-y)\mathrm{d}y,\\
&&\dot{S}_ju=\chi(2^{-j}D)u.
\end{eqnarray*}
\begin{definition}
We denote by $\mathcal{S}_h'$ the space of temperate distributions $u$ such that
$$\lim_{j\rightarrow-\infty}\dot{S}_j u=0 \ \ \text{in}\ \ \ \mathcal{S}'.$$
\end{definition}
\begin{remark}\label{Sh}
If a temperate distribution $u$ is such that its Fourier transform $\mathcal{F}u$ is locally integrable near $0$, then $u$ belongs to $\mathcal{S}_h'$.
\end{remark}
Then the homogeneous Littlewood-paley decomposition is defined as
\begin{equation}\label{decomposition}
u=\sum_{j\in\mathbb{Z}}\dot{\Delta}_ju,\qquad \mbox{for}\quad u\in\mathcal{S}'_h.
\end{equation}
 With our choice of $\varphi$ and $\chi$, it is easy to verify that
\begin{equation}\label{cancellations}
\dot{\Delta}_j\dot{\Delta}_k u=0\ \ \text{if}\ \ |j-k|\geq2,\ \  \ \text{and} \ \ \ \  \dot{\Delta}_j(\dot{S}_{k-1}u\dot{\Delta}_ku)=0\ \ \text{if}\ \ |j-k|\geq5.
\end{equation}
Next, let us introduce a useful lemma which will be repeatedly used throughout this paper.
\begin{lemma}\label{benstein}
\cite{H.BAHOURI} Let $1\leq p\leq q\leq+\infty$. Then for any
$\gamma\in(\mathbb{N}\cup{\{0\}})^d$, there exists a constant $C$
independent of $f,j$ such that
\begin{equation*}\begin{aligned}
&\mathrm{supp}\hat{f}\subseteq\{|\xi|\leq A_02^j\}\Rightarrow
\|\partial^\gamma f\|_{L^q}\leq C2^{j|\gamma|+jd(\frac{1}{p}-\frac{1}{q})}\|f\|_{L^p},\\
&\mathrm{supp}\hat{f}\subseteq\{A_1 2^j\le|\xi|\leq A_22^j\}\Rightarrow\|f\|_{L^p}\leq C2^{-j|\gamma|}\sup_{|\beta|=|\gamma|}\|\partial^\beta f\|_{L^p}.
\end{aligned}\end{equation*}
\end{lemma}

\subsection{Homogeneous Besov space}
\begin{definition}\label{Besovdefinition}
Let $u\in\mathcal{S}'(\mathbb{R}^d)$, $s\in\mathbb{R}$, and $1\leq p,r\leq\infty$, we set
$$\|u\|_{\dot{B}^s_{p,r}}\triangleq \|\{ 2^{js}\|\dot{\Delta}_j u\|_{L^p}\}_j\|_{l^r}.$$
We then define the space $\dot{B}^s_{p,r}\triangleq\{u\in\mathcal{S}'_h,\, \|u\|_{\dot{B}^s_{p,r}}<\infty\}.$
\end{definition}
\begin{remark}\label{wellposedness}
The definition of the space $\dot{B}^s_{p,r}$ does not depend on the choice of the couple ($\varphi,\chi$) defining the Littlewood-Paley decomposition.
\end{remark}

Let us now state some classical properties of the homogeneous Besov spaces.
\begin{proposition}\label{Besovproperties}
For all $s,s_1,s_2\in\mathbb{R},1\leq p,p_1,p_2,r,r_1,r_2\leq+\infty$, we have the following properties:  \\
(i) If $p_1\leq p_2$, $r_1\leq r_2$, then $\dot{B}^s_{p_1,r_1}\hookrightarrow \dot{B}^{s-\frac{d}{p_1}+\frac{d}{p_2}}_{p_2,r_2}$.\\
(ii) If $s_1\neq s_2$ and $\theta\in(0,1)$, then
$$\|u\|_{\dot{B}_{p,r}^{\theta s_1+(1-\theta)s_2}}\leq\|u\|^\theta_{\dot{B}^{s_1}_{p,r}}\|u\|^{1-\theta}_
{\dot{B}^{s_2}_{p,r}}.$$
(iii) $\dot{H}^s\thickapprox\dot{B}^s_{2,2}$ and
$$\frac{1}{C^{|s|+1}}\|u\|_{\dot{B}^s_{2,2}}\leq\|u\|_{\dot{H}^s}\leq C^{|s|+1}\|u\|_{\dot{B}^s_{2,2}}.$$
(iv) If $s>0$, then $\dot{B}^s_{2,1}\cap L^\infty$ (especially $\dot{B}^{\frac{d}{2}}_{2,1}$) is an algebra.
\end{proposition}
\begin{proposition}\label{BONY}
Let $s>0, u\in L^\infty\cap\dot{B}^s_{2,1}$ and $v\in L^\infty\cap\dot{B}^s_{2,1}$. Then $uv\in L^\infty\cap\dot{B}^s_{2,1}$ and
$$\|uv\|_{\dot{B}^s_{2,1}}\lesssim \|u\|_{L^\infty}\|v\|_{\dot{B}^s_{2,1}}+\|v\|_{L^\infty}\|u\|_{\dot{B}^s_{2,1}}.$$
Let $s_1,s_2\leq\frac{d}{2}$ such that $s_1+s_2>0$, $u\in\dot{B}^{s_1}_{2,1}$ and $v\in\dot{B}^{s_2}_{2,1}$. Then $uv\in\dot{B}^{s_1+s_2-\frac{d}{2}}_{2,1}$ and
$$\|uv\|_{\dot{B}^{s_1+s_2-\frac{d}{2}}_{2,1}}\lesssim\|u\|_{\dot{B}^{s_1}
_{2,1}}\|v\|_{\dot{B}^{s_2}_{2,1}}.$$
Let $u\in\dot{B}^0_{d,\infty}$
and $v\in\dot{B}^1_{d,1}$. Then $uv \in\dot{B}^0_{d,\infty}$
 $$\|uv\|_{\dot{B}^0_{d,\infty}}\lesssim\|u\|_{\dot{B}^0_{d,\infty}}
\|v\|_{\dot{B}^1_{d,1}}.$$
\end{proposition}
We can refer to \cite{H.BAHOURI} for the proof of these
propositions.

\begin{proposition}\label{example}
Let $\sigma\in(1,\frac{3}{2}]$, then we  get $$\chi_p(D)|x|^{-\sigma}\in{\dot{B}^{\frac{1}{2}}_{2,1}(\mathbb{R}^3)};\ \ \chi_p(D)|x|^{-\sigma}\notin H^2(\mathbb{R}^3).$$
Here $\chi_p(D)f=\mathcal{F}^{-1}(\chi_{B(0,p)}\hat{f})$.
\end{proposition}

\begin{proof}

Noticing that
$$\mathcal{F}(|x|^{-\sigma})=C|\xi|^{\sigma-3}\in L^1_{loc}(\mathbb{R}^3)$$ and Remark \ref{Sh}, we
get
 $|x|^{-\sigma}\in \mathcal{S}'_h$.
By direct computations we have
\begin{equation*}
\begin{split}
\dot{\Delta}_j|x|^{-\sigma}&=2^{3j}\int_{\mathbb{R}^3}h\big(2^j(x-y)\big)
|y|^{-\sigma}\mathrm{d}y \\
&=2^{j\sigma}\int_{\mathbb{R}^3}h(2^jx-2^jy)
|2^jy|^{-\sigma}\mathrm{d}2^jy  \\
&=2^{j\sigma}\widetilde{h}(2^jx),
\end{split}
\end{equation*}
where $\widetilde{h}(x)=\int_{\mathbb{R}^3}h(x-y)
|y|^{-\sigma}\mathrm{d}y$.
Since $(\mathcal{F}h)(\xi)=\varphi(\xi)\in\mathcal{D}(\mathbb{R}^3/\{0\})$, we obtain $\mathcal{F}\widetilde{h}\in\mathcal{D}(\mathbb{R}^3)$. This implies $\|\widetilde{h}\|_{L^2}<+\infty$.
From above analysis, we easily get
$$\|\dot{\Delta}_j|x|^{-\sigma}\|_{L^2}=2^{j(\sigma-\frac{3}{2})}\|
\widetilde{h}\|_{L^2}.$$ Combining with the definition of the Besov space and $\mathrm{supp}\chi_p\subset B(0,p)$, we derive that
\begin{equation*}
\begin{split}
\|\chi_p(D)|x|^{-\sigma}\|_{\dot{B}^{\frac{1}{2}}_{2,1}}\lesssim\sum_{j\leq N_p} 2^{j(\sigma-1)}\|\widetilde{h}\|_{L^2}<\infty.
\end{split}
\end{equation*}
However, it is obvious that
\begin{equation*}
\|\chi_p(D)|x|^{-\sigma}\|_{L^2}=
\|\chi_p(\xi)|\xi|^{\sigma-3}\|_{L^2}=\infty.
\end{equation*}
\end{proof}

\subsection{Hybrid Besov spaces}\label{Hybrid}

Let us now introduce the hybrid Besov spaces we will work in this paper.
\begin{definition}\label{define of hybrid}
\cite{R.DANCHIN2001} Let $s,t\in\mathbb{R}$. We set
$$\|u\|_{\dot{B}^{s,t}}=\sum_{j\leq0}2^{js}\|\dot{\Delta}_ju\|_{L^2}
+\sum_{j>0}2^{jt}\|\dot{\Delta}_ju\|_{L^2}.$$
We then define the space $\dot{B}^{s,t}\triangleq\{u\in\mathcal{S}'_h,\, \|u\|_{\dot{B}^{s,t}}<\infty\}.$
\end{definition}
\noindent We will often use the definition:
$$\|u\|^l_{\dot{B}^s_{p,1}}\triangleq\sum_{j\leq0}2^{js}\|\dot{\Delta}_ju\|_p;\ \ \ \ \|u\|^h_{\dot{B}^s_{p,1}}\triangleq\sum_{j>0}2^{js}\|\dot{\Delta}_ju\|_p.$$
We also define the norm of the space $L^r_T(\dot{B}^{s,t})$
\begin{equation*}
\begin{split}
&\|u\|_{L^r_T(\dot{B}^{s,t})}=\Big(\int_0^T\|u\|_{\dot{B}^{s,t}}^r
\mathrm{d}t\Big)^{\frac{1}{r}},
\end{split}
\end{equation*}with the usual change if $r=\infty$.
Furthermore, the norm of the time-space Besov space
$\widetilde{L}^r_T(\dot{B}^{s,t})$ is defined by
\begin{equation*}
\begin{split}
&\|u\|_{\widetilde{L}^r_T(\dot{B}^{s,t})}=\sum_{j\leq 0}2^{js}\|\dot{\Delta}_ju\|_{L^r_TL^2}+\sum_{j> 0}2^{js}\|\dot{\Delta}_ju\|_{L^r_TL^2}.
\end{split}
\end{equation*}
\begin{remark}\label{hua}
(i) \ By the Minkowski inequality, we easily find that $\widetilde{L}^1_T(\dot{B}^{s,t})=L^1_T(\dot{B}^{s,t})$ and $\widetilde{L}^r_T(\dot{B}^{s,t})\subseteq L^r_T(\dot{B}^{s,t})$ for $r>1$.\\
(ii)\  The properties of Proposition \ref{BONY} for product remain
true for the time-space Besov spaces.
\end{remark}
The following Proposition is a direct consequence of the definition of Besov space.
\begin{proposition}\cite{R.DANCHIN2001}
(i)\ We have $\dot{B}^{s,s}=\dot{B}^s_{2,1}$.\\
(ii) If $s\leq t$ then $\dot{B}^{s,t}=\dot{B}^s_{2,1}\cap\dot{B}^t_{2,1}$. Otherwise, $\dot{B}^{s,t}=\dot{B}^s_{2,1}+\dot{B}^t_{2,1}$.\\
(iii) If $s_1\leq s_2$ and $t_1\geq t_2$, then $\dot{B}^{s_1,t_1}\hookrightarrow\dot{B}^{s_2,t_2}$.
\end{proposition}
For the estimates of the product of two temperate distributions $u$
and $v$ in hybrid Besov spaces, we refer to the Appendix in section
\ref{Appendix}.
\subsection{Smoothing properties for the linear heat equation}
\begin{proposition}\label{smoothing}
\cite{H.BAHOURI} Let $s\in\mathbb{R}, (p,r)\in[1,+\infty]^2$,
$u_0\in\dot{B}^{s-1}_{p,r}$ and
$f\in\widetilde{L}^1_T(\dot{B}^{s-1}_{p,r})$. Let $u$ solve
\begin{equation}
\begin{cases}
\partial_t u -\mu\Delta u=f,\\
u(0,x)=u_0.
\end{cases}
\end{equation}
Then $u\in \widetilde{L}^{\infty}_T(\dot{B}^{s-1}_{p,r})\cap \widetilde{L}^1_T(\dot{B}^{s+1}_{p,r})$ and the following estimate holds:
$$\|u\|_{\widetilde{L}^\infty_T(\dot{B}^{s-1}_{p,r})}+
\mu\|u\|_{\widetilde{L}^1_T(\dot{B}^{s+1}_{p,r})}\leq\|u_0\|_{\dot{B}^{s-1}_{p,r}}
+C\|f\|_{\widetilde{L}^1_T(\dot{B}^{s-1}_{p,r})}.$$
If in addition $r<+\infty$, then $u\in C_b([0,T);\dot{B}^{s-1}_{p,r})$.
\end{proposition}
The following two Propositions are used for the proof of the
uniqueness.
\begin{proposition}\label{transport}
\cite{H.BAHOURI} Let $(p,r)\in [1,+\infty]^2$ and
$s\in\big(-\min(\frac{d}{p},\frac{d}{p'});1+\frac{d}{p}\big)$. Let
$u$ be a vector field such that $\nabla u$ belongs to
$L^1(0,T;\dot{B}^{\frac{d}{p}}_{p,r}\cap L^{\infty})$. Suppose that
$f_0\in\dot{B}^s_{p,r}, F\in L^1(0,T;\dot{B}^s_{p,r})$ and that
$f\in L^{\infty}(0,T;\dot{B}^s_{p,r})\cap C([0,T];\mathcal{S}')$
solves
\begin{equation*}
\begin{cases}
\partial_tf+u\cdot\nabla f=F,\\
f(0,x)=f_0.
\end{cases}
\end{equation*}
Let $V(t)\triangleq\int_0^t\|\nabla u\|_{\dot{B}^{\frac{d}{p}}_{p,r}\cap L^{\infty}}\mathrm{d}t'$. There exists a constant $C$ depending only on $s,p$ and $d$, and such that the following inequality holds true for $t\in[0,T]$
\begin{equation}
\|f\|_{\widetilde{L}_t^\infty(\dot{B}^s_{p,r})}\leq e^{CV(t)}\Big( \|f_0\|_{\dot{B}^s_{p,r}}+\int_0^t e^{-CV(t)}\|F(t')\|_{\dot{B}^s_{p,r}}\mathrm{d}t'\Big).
\end{equation}
If $r<+\infty$ then $f$ belongs to $C([0,T];\dot{B}^s_{p,r})$.
\end{proposition}

\begin{proposition}\label{log}
\cite{Chen,R.DANCHIN2005} Let $s\in\mathbb{R}$. Then for any $1\leq
p,r\leq+\infty$ and $0<\epsilon\leq1$, we have
$$\|f\|_{\widetilde{L}^r_T(\dot{B}^s_{p,1})}\leq C\frac{\|f\|_{\widetilde{L}^r_T(\dot{B}^s_{p,\infty})}}{\epsilon}
\log\Big( e+\frac{\|f\|_{\widetilde{L}^r_T(\dot{B}^{s-\epsilon}_{p,\infty})}+
\|f\|_{\widetilde{L}^r_T(\dot{B}^{s+\epsilon}_{p,\infty})}}{\|f\|_{\widetilde{L}^r_T
(\dot{B}^s_{p,\infty})}}  \Big).$$
\end{proposition}

\section{Priori estimate}\label{PRIORI}

\setcounter{section}{3}\setcounter{equation}{0}

Recalling \eqref{el} and \eqref{eh}, we denote
\begin{equation}\label{energytotal}
\mathcal{E}(t)\triangleq\mathcal{E}^l(t)+\mathcal{E}^h(t).
\end{equation} We  have the following priori estimate.
\begin{proposition}\label{pioriestimate}
Let $d\geq2$ and $a=0$. Assume that the system \eqref{OBA} has a solution $(\tau,u)$ on $[0,T)$. Then, there exist two positive constants $C_1, C_2$ independent of $T$ such that
\begin{equation}\label{DD}
\mathcal{E}(t)\leq C_1\mathcal{E}_0+C_2\mathcal{E}^2(t),
\end{equation}
where
$\mathcal{E}_0=\|u_0\|_{\dot{B}^{\frac{d}{2}-1}_{2,1}}+\|\tau_0\|_
{\dot{B}^{\frac{d}{2}-1,\frac{d}{2}}}$.

\end{proposition}

Before the proof of the priori estimate, let us introduce two important Lemmas.

First, we show the following rough energy estimate in frequency
spaces. One refers to \eqref{elr} and \eqref{ehr} for the
definitions of $\mathcal{E}^h_r(t)$ and $\mathcal{E}^l_r(t)$.

\begin{lemma}\label{originalLemma}
Let $d\geq2, a=0$ and $(u,\tau)$ be the solution of system \eqref{OBA} on $[0,T)$.
The following estimate holds
\begin{equation*}
\begin{split}
\mathcal{E}^l_r(t)+\mathcal{E}^h_r(t)\leq& C_1\mathcal{E}_0+ C_2\int_0^t \sum_{j\leq 0}2^{j(\frac{d}{2}-1)}\tilde{E}_j/(\|\dot{\Delta}_j u\|_{L^2}+\|\dot{\Delta}_j\Lambda^{-1}\mathbb{P}\nabla\cdot\tau\|_{L^2})\mathrm{d}t'\\
&+C_2\int_0^t\sum_{j>0} 2^{j(\frac{d}{2}-1)}\Big(\tilde{F}_j/(\|\dot{\Delta}_j u\|_{L^2}+\|\dot{\Delta}_j\mathbb{P}\nabla\cdot\tau\|_{L^
2})+\tilde{R}_j/\|\dot{\Delta}_ju\|_{L^2}\Big)\mathrm{d}t',
\end{split}
\end{equation*}
where $C_1, C_2$ independent of $T$.
\end{lemma}
Here we shall omit the definitions of $\tilde{E}_j, \tilde{F}_j, \tilde{R}_j$ for their specific forms do not affect the proof of Lemma \ref{originalLemma}. One can refer to the proof of Proposition \ref{pioriestimate} for their definitions.
\begin{proof}
Let us first fix a constant $j_0\in\mathbb{Z}$, which will be chosen in Step $2$. Throughout this part we shall suppose that $j\leq j_0$ as low frequencies and $j>j_0$ as high frequencies. Moreover, without loss of generality, we assume that
$$\|\dot{\Delta}_j u\|_{L^2}, \ \|\dot{\Delta}_j\Lambda^{-1}\mathbb{P}\nabla\cdot\tau\|_{L^2},\  \|\dot
{\Delta}_j\mathbb{P}\nabla\cdot\tau\|_{L^2}\neq 0. $$

\vskip 0.1cm
\textit{\bf{Step 1: Estimate of $\mathcal{E}^l_r(t)$.}}\label{step1}

\vskip 0.1cm

Set $$\widetilde{C}_0=2^{j_0}.$$  Thanks to Lemma \ref{benstein}, for any function $f$, there exist two constants $\widetilde{C}_1$, $\widetilde{C}_2$ such that
\begin{equation}\label{Bern}
\widetilde{C}_12^j\|\dot{\Delta}_jf\|_{L^p}\leq\|\Lambda \dot{\Delta}_jf\|_{L^p}\leq  \widetilde{C}_22^j\|\dot{\Delta}_jf\|_{L^p}.
\end{equation}
Applying the operator $\dot{\Delta}_j$ to the system \eqref{OBLG}, we get
\begin{equation}\label{equationut}
\begin{cases}
(\dot{\Delta}_ju)_t+\nu \Lambda^2\dot{\Delta}_ju-\mu_1\Lambda\dot{\Delta}_j\Lambda^{-1}\mathbb{P}\nabla\cdot\tau=
\dot{\Delta}_j\mathbb{P} E,\\
(\dot{\Delta}_j\Lambda^{-1}\mathbb{P}\nabla\cdot\tau)_t+\frac{\mu_2}{2}\Lambda\dot{\Delta}_j u=\dot{\Delta}_j\Lambda^{-1}\mathbb{P}\nabla\cdot F, \\
\end{cases}
\end{equation}
with
$$E=-u\cdot\nabla u ;\ \ \ \  F=-u\cdot\nabla\tau-Q(\tau,\nabla u).$$
For the sake of simplicity, we first suppose that $E, F=0$. Taking the $L^2$ scalar product of the first equation of \eqref{equationut} with $\dot{\Delta}_ju$, of the second with $\dot{\Delta}_j\Lambda^{-1}\mathbb{P}\nabla\cdot\tau$, we obtain the following two identities:
\begin{eqnarray}\label{ul}
&&\frac{1}{2}\frac{d}{dt}\|\dot{\Delta}_ju\|_{L^2}^2+\nu\|\Lambda \dot{\Delta}_j u\|^2_{L^2}-\mu_1(\Lambda\dot{\Delta}_j(\Lambda^{-1}\mathbb{P}\nabla\cdot\tau),
\dot{\Delta}_ju)=0,\\
\label{taul}
&&\frac{1}{2}\frac{d}{dt}\|\dot{\Delta}_j(\Lambda^{-1}\mathbb{P}\nabla\cdot\tau)\|^2_{L^2}
+\frac{\mu_2}{2}
(\Lambda\dot{\Delta}_ju,\dot{\Delta}_j\Lambda^{-1}\mathbb{P}\nabla\cdot\tau)=0.
\end{eqnarray}
To obtain an identity involving $(\dot{\Delta}_j\mathbb{P}\nabla\cdot\tau,\dot{\Delta}_j u)$, we take the scalar product of the first equation of \eqref{equationut} with $\dot{\Delta}_j\mathbb{P}\nabla\cdot\tau$, apply $\Lambda$ to the second equation and take the $L^2$ scalar product with $\dot{\Delta}_ju$, then sum up both equalities, which produces
\begin{equation}\label{tul}
\begin{split}
\frac{d}{dt}(\dot{\Delta}_j\mathbb{P}\nabla\cdot\tau,\dot{\Delta}_j u)+\nu(\dot{\Delta}_j\Lambda^2u,\dot{\Delta}_j\mathbb{P}\nabla\cdot\tau)-\mu_1
\|\dot{\Delta}_j\mathbb{P}\nabla\cdot\tau\|^2_{L^2}+\frac{\mu_2}{2}\|
\Lambda\dot{\Delta}_j u\|^2_{L^2}=0.
\end{split}
\end{equation}
Calculating $\frac{\mu_2}{2}\eqref{ul}+\mu_1\eqref{taul}-K_1\eqref{tul}$, we obtain
\begin{equation}\label{addl}
\begin{split}
&\frac{1}{2}\frac{d}{dt}\Big(\frac{\mu_2}{2}\|\dot{\Delta}_j u\|^2_{L^2}+\mu_1\|\dot{\Delta}_j\Lambda^{-1}\mathbb{P}\nabla\cdot\tau\|^2_{L^2}-2K_1(\dot{\Delta}_j\mathbb{P}\nabla\cdot\tau,\dot{\Delta}_j u)\Big) \ \ \ \ \ \\
&+\big(\frac{\mu_2\nu}{2}-\frac{\mu_2K_1}{2}\big)\|\Lambda\dot{\Delta}_j u\|^2_{L^2}+\mu_1K_1\|\dot{\Delta}_j\mathbb{P}\nabla\cdot\tau\|^2_{L^2}
-\nu K_1(\dot{\Delta}_j\Lambda^2u,\dot{\Delta}_j\mathbb{P}\nabla\cdot\tau)\\
&=0.
\end{split}
\end{equation}
Using H\"{o}lder's inequality, Young's inequality and \eqref{Bern}, we find that for all $\epsilon_0, \epsilon_1>0$,
\begin{eqnarray*}
&&2K_1|(\dot{\Delta}_j\mathbb{P}\nabla\cdot\tau,\dot{\Delta}_j u)|\leq \frac{1}{\epsilon_0}\|\dot{\Delta}_ju\|^2_{2}+\epsilon_0\widetilde{C}_{2}^2\widetilde{C}_0^2K_1^2
\|\dot{\Delta}_j\Lambda^{-1}\mathbb{P}\nabla\cdot\tau\|^2_{L^2},\\
&&\nu K_1|(\dot{\Delta}_j\Lambda^2u,\dot{\Delta}_j\mathbb{P}\nabla\cdot\tau)|\leq\frac{\nu \widetilde{C}_0^2\widetilde{C}_2^2K_1}{2\epsilon_1}\|\dot{\Delta}_j\Lambda u\|^2_{L^2}+\frac{\epsilon_1\nu K_1}{2}\|\dot{\Delta}_j\mathbb{P}\nabla\cdot\tau\|_{L^2}^2.
\end{eqnarray*}
Inserting the above two inequalities into \eqref{addl}  with $\epsilon_0=\frac{4}{\mu_2}, \epsilon_1=\frac{\mu_1}{\nu}$, and combining \eqref{Bern}, we derive
\begin{equation}\label{UTL}
\begin{split}
&\frac{1}{2}\frac{d}{dt}\bigg(\frac{\mu_2}{4}\|\dot{\Delta}_ju\|_{L^2}^2+
\Big(\mu_1-\frac{4\widetilde{C}
_0^2\widetilde{C}_{2}^2K_1^2}{\mu_2}\Big)
\|\dot{\Delta}_j\Lambda^{-1}\mathbb{P}\nabla\cdot\tau\|_{L^2}^2\bigg)\\
&+\big(\frac{\mu_2\nu}{2}-\frac{\mu_2K_1}{2}-\frac{\nu^2
\widetilde{C}_0^2\widetilde{C}_{2}^2K_1}{2\mu_1}\big)\widetilde{C}^2_{1}2^{2j}
\|\dot{\Delta}_ju\|_{L^2}^2+
\frac{\mu_1\widetilde{C}_{1}^2K_1}{2}2^{2j}\|\dot{\Delta}_j{\Lambda}^{-1}
\mathbb{P}\nabla\cdot\tau\|_{L^2}^2 \\
& \leq 0.
\end{split}
\end{equation}
Choosing a constant $K_1$ such that $0<K_1<\min\big\{\frac{\sqrt{\mu_1\mu_2}}{2\widetilde{C}_0
\widetilde{C}_2},\frac{\mu_1\mu_2\nu}{\mu_1\mu_2+\nu^2\widetilde{C}_0^2
\widetilde{C}_{2}^2}\big\}$,
we ensure that the coefficients of $\|\dot{\Delta}_j{\Lambda}^{-1}\mathbb{P}\nabla\cdot\tau\|_{L^2}^2$ and $\|\dot{\Delta}_ju\|_{L^2}^2$ in \eqref{UTL} are positive.
Thus in the general case ($E, F\neq0$) that containing the nonlinear terms,
we obtain
\begin{equation}\label{aa}
\begin{split}
&\frac{1}{2}\frac{d}{dt}\Big(\|\dot{\Delta}_ju\|_{L^2}^2+
\|\dot{\Delta}_j\Lambda^{-1}\mathbb{P}\nabla\cdot\tau\|_{L^2}^2\Big)
+2^{2j}\Big(\|\dot{\Delta}_ju\|_{L^2}^2+
\|\dot{\Delta}_j{\Lambda}^{-1}\mathbb{P}\nabla\cdot\tau\|_{L^2}^2\Big
)
\\&\leq C_2 \tilde{E}_j,
\end{split}
\end{equation}
for the constant $C_2>0$.
Dividing by $\|\dot{\Delta}_ju\|_{L^2}+
\|\dot{\Delta}_j\Lambda^{-1}\mathbb{P}\nabla\cdot\tau\|_{L^2}$, \eqref{aa} can be written as
\begin{equation}\label{UL}
\begin{split}
&\frac{d}{dt}\Big(\|\dot{\Delta}_ju\|_{L^2}+
\|\dot{\Delta}_j\Lambda^{-1}\mathbb{P}\nabla\cdot\tau\|_{L^2}\Big)
+2^{2j}\Big(\|\dot{\Delta}_ju\|_{L^2}+
\|\dot{\Delta}_j{\Lambda}^{-1}\mathbb{P}\nabla\cdot\tau\|_{L^2} \Big ) \\
&\leq C_2 \tilde{E}_j/(\|\dot{\Delta}_ju\|_{L^2}+
\|\dot{\Delta}_j\Lambda^{-1}\mathbb{P}\nabla\cdot\tau\|_{L^2}).
\end{split}
\end{equation}
Multiplying both sides of \eqref{UL} by $2^{j(\frac{d}{2}-1)}$, summing up by $j\leq j_0$ (actually we can choose $j_0=0$, see Step 2), and then performing the time integration over $[0,t]$, we conclude that
\begin{equation}\label{e0l}
\begin{split}
\mathcal{E}^l_r(t)\leq &C_1( \|u_0\|^l_{\dot{B}^{\frac{d}{2}-1}_{2,1}}+\|\tau_0\|^l_{\dot{B}^{\frac{d}{2}
-1}_{2,1}}) \\
&+C_2 \int_0^t \sum_{j\leq j_0}2^{j(\frac{d}{2}-1)}\tilde{E}_j/(\|\dot{\Delta}_j u\|_{L^2}+\|\dot{\Delta}_j\Lambda^{-1}\mathbb{P}\nabla\cdot\tau\|_{L^2})\mathrm{d}t'.
\end{split}
\end{equation}
\vskip 0.1cm
\textit{\bf{Step 2: Estimate of $\mathcal{E}^h_r(t)$.}}\label{step2}
\vskip 0.1cm
By the similar argument as the way in the low frequencies, we first consider the linear terms of \eqref{equationut} as well. From the second equation of \eqref{equationut}, we can easily obtain the following inequality:
\begin{equation}\label{tauh}
\frac{1}{2}\frac{d}{dt}\|\dot{\Delta}_j\mathbb{P}\nabla\cdot\tau\|_{L^2}^2
+\frac{\mu_2}{2}
(\dot{\Delta}_j\Lambda^2u,\dot{\Delta}_j\mathbb{P}\nabla\cdot\tau)=0.
\end{equation}
Calculating $\nu\eqref{tauh}-\frac{\mu_2}{2}\eqref{tul}+K_2\eqref{ul}$, we derive
\begin{equation}\label{TUH}
\begin{split}
&\frac{1}{2}\frac{d}{dt}\Big(K_2\|\dot{\Delta}_ju\|_{L^2}^2+\nu\|\dot{\Delta}
_j\mathbb{P}\nabla\cdot\tau\|_{L^2}^2-\mu_2(\dot{\Delta}_j
\mathbb{P}\nabla\cdot\tau,\dot{\Delta}_j u)\Big)\\
&+\big(\nu K_2-\frac{\mu_{2}^2}{4}\big)\|\Lambda\dot{\Delta}_ju\|_{L^2}^2+\frac{\mu_1\mu_2}
{2}\|\dot{\Delta}_j\mathbb{P}\nabla\cdot\tau\|_{L^2}^2
-\mu_1K_2(\dot{\Delta}_j\mathbb{P}\nabla\cdot\tau,\dot{\Delta}_ju)=0.
\end{split}
\end{equation}
It is easy to check that for all $\epsilon_0, \epsilon_1>0$, we have
\begin{equation}\label{1}\begin{split}
&\mu_2|(\dot{\Delta}_j\mathbb{P}\nabla\cdot\tau,\dot{\Delta}_j u)|\leq\frac{\epsilon_0}{2}\|\dot{\Delta}_ju\|^2_2+\frac{\mu_2^2}{2\epsilon_0}\|\dot
{\Delta}_j\mathbb{P}\nabla\cdot\tau\|^2_{L^2},\\
&\mu_1K_2|(\dot{\Delta}_j\mathbb{P}\nabla\cdot\tau,\dot{\Delta}_ju)|\leq
\frac{\epsilon_1\mu_1K_2}{2}\|\dot{\Delta}_ju\|^2_{L^2}+
\frac{\mu_1K_2}{2\epsilon_1}\|\dot{\Delta}_j\mathbb{P}\nabla\cdot\tau\|^2_{L^2}.
\end{split}
\end{equation}
Noticing that $j\geq j_0+1$, the following inequality shows that the smoothing effect of system \eqref{OBLG} on $u$ will lose,
\begin{equation}\label{2}
\|\Lambda\dot{\Delta}_ju\|^2_{L^2}\geq \widetilde{C}^2_12^{2j}\|\dot{\Delta}_ju\|^2_{L^2}\geq \widetilde{C}^2_1\widetilde{C}^2_0\|\dot{\Delta}_j u\|^2_{L^2}.
\end{equation}
Combining $\eqref{TUH}\sim\eqref{2}$, and choosing
$$K_2=\epsilon_0=\frac{\mu^2_{2}}{\nu};\ \ \ \epsilon_1=\frac{2K_2}{\mu_2}.$$
we finally get
\begin{equation}\label{tuh}
\begin{split}
&\frac{1}{2}\frac{d}{dt}\Big(\frac{\mu^2_{2}}{2\nu}\|\dot{\Delta}_ju\|_{L^2}
^2+\frac{\nu}{2}\|\dot{\Delta}_j\mathbb{P}\nabla\cdot\tau\|^2_{L^2}\Big)  \\
&+\big(\frac{3}{4}\mu_2^2\widetilde{C}_1^2\widetilde{C}_0^2-\frac{\mu_1\mu_2^3}{\nu^2}\big)
\|\dot{\Delta}_ju\|_{L^2}^2+\frac{\mu_1\mu_2}{4}\|\dot{\Delta}_j
\mathbb{P}\nabla\cdot\tau\|_{L^2}^2=0.
\end{split}
\end{equation}
We then define $j_0=[\log_2\frac{2\sqrt{3\mu_1\mu_2}}{3\nu \widetilde{C}_1}]+1$ to ensure that the coefficient of $\|\dot{\Delta}_ju\|_2^2$ in \eqref{tuh} is positive. In light of the Remark \ref{wellposedness}, without loss of generality, we can define $j_0=0$  here and later.
In the general case where the nonlinear terms $E, F\neq0$, we get
\begin{equation*}
\begin{split}
&\frac{1}{2}\frac{d}{dt}\Big(\|\dot{\Delta}_ju\|_{L^2}^2+\|\dot{\Delta}_j\mathbb{P}\nabla\cdot
\tau\|^2_{L^2}\Big) +(\|\dot{\Delta}_ju\|_{L^2}^2+\|\dot{\Delta}_j\mathbb{P}\nabla\cdot\tau\|_{L^2}^2)\leq C_2\tilde{F}_j.
\end{split}
\end{equation*}
It is easy to see that the above estimate can be rewritten as
\begin{equation}\label{UTauL}
\begin{split}
\frac{d}{dt}\Big(\|\dot{\Delta}_ju\|_{L^2}+&\|\dot{\Delta}_j\mathbb{P}\nabla\cdot
\tau\|_{L^2}\Big) +(\|\dot{\Delta}_ju\|_{L^2}+\|\dot{\Delta}_j\mathbb{P}\nabla\cdot\tau\|_{L^2})\\
&\leq C_2\tilde{F}_j/(\|\dot{\Delta}_ju\|_{L^2}+\|\dot{\Delta}_j\mathbb{P}\nabla\cdot
\tau\|_{L^2}).
\end{split}
\end{equation}
Multiplying both sides of \eqref{UTauL} by $2^{j(\frac{d}{2}-1)}$, then summing up by $j> j_0$, we deduce after performing the time integration over $[0,t]$, that
\begin{equation}\label{e0h}
\begin{split}
&\sup_t\|u\|^h_{\dot{B}^{\frac{d}{2}-1}_{2,1}}+
\int_0^t+\sup_t\|\mathbb{P}\nabla\cdot\tau\|^h_{\dot{B}^{\frac{d}{2}-1}_{2,1}}+
\int_0^t\|u\|^h_{\dot{B
}^{\frac{d}{2}-1}_{2,1}}\mathrm{d}t'+
\|\mathbb{P}\nabla\cdot\tau\|^h_{\dot{B}^{\frac{d}{2}-1}_{2,1}}\mathrm{d}t'\\&\lesssim \|u_0\|^h_{\dot{B}^{\frac{d}{2}-1}_{2,1}}+\|\tau_0\|^h_{\dot{B}^{\frac{d}{2}}_{2,1}}+\int_0^t \sum_{j> 0}2^{j(\frac{d}{2}-1)}\tilde{F}_j/(\|\dot{\Delta}_j u\|_{L^2}+\|\dot{\Delta}_j\mathbb{P}\nabla\cdot\tau\|_{L^2})\mathrm{d}t'.
\end{split}
\end{equation}

Next we show the smoothing effect on $u$. In the same way as  the derivation of \eqref{ul}, we get
\begin{equation*}
\frac{d}{dt}\|\dot{\Delta}_ju\|_{L^2}+2^{2j}\|\dot{\Delta}_ju\|_{L^2}\leq C\|\dot{\Delta}_j\mathbb{P}\nabla\cdot\tau\|_{L^2}+C\tilde{R}_j/\|\dot{\Delta}_ju\|_{L^2}.
\end{equation*}
Using the above inequality we deduce
\begin{equation}\label{uh}
\begin{split}
\sup_t\|u\|^h_{\dot{B}^{\frac{d}{2}-1}_{2,1}}+&\int_0^t\|u\|^h_{\dot{B
}^{\frac{d}{2}+1}_{2,1}}\mathrm{d}t'
\leq \|u_0\|^h_{\dot{B}^{\frac{d}{2}-1}_{2,1}}+ C\int_0^t\|\mathbb{P}\nabla\cdot\tau\|^h_{\dot{B}^{\frac{d}{2}-1}_{2,1}}\mathrm{d}t'\\
&+C\int_0^t\sum_{j>0}2^{j(\frac{d}{2}-1)}\tilde{R}_j/\|\dot{\Delta}_ju\|_{L^2}\mathrm{d}t'.
\end{split}
\end{equation}
Calculating $\eqref{e0h}+\eta_1\eqref{uh}$, and if we choose $\eta_1$ small enough such that $\eta_1C\leq\frac{1}{2}$, then the term $\eta_1C\int_0^t\|\mathbb{P}\nabla\cdot\tau\|^h_{\dot{B}^{\frac{d}{2}-1}_{2,1}}\mathrm{d}t'$ can be absorbed by the left side. Therefore, we obtain
\begin{equation}\label{uhh}
\begin{split}
\mathcal{E}^h_r(t)\leq& C_1\big(\|u_0\|^h_{\dot{B}^{\frac{d}{2}-1}_{2,1}}+ \|\tau_0\|^h_{\dot{B}^{\frac{d}{2}}_{2,1}} \big) \\
&+C_2\int_0^t\sum_{j>0}2^{j(\frac{d}{2}-1)}\Big(\tilde{F}_j/(\|\dot{\Delta}_ju\|_{L^2}+\|\dot
{\Delta}_j\mathbb{P}\nabla\cdot\tau\|_{L^2})+\tilde{R}_j/\|\dot{\Delta}_ju\|
_{L^2}\Big)\mathrm{d}t'.
\end{split}
\end{equation}

Combining \eqref{e0l} and \eqref{uhh}, we complete the proof of Lemma \ref{originalLemma}.
\end{proof}
Although we have derived the above rough energy estimate, it seems that some parts of the nonlinear terms could not be controlled by $\mathcal{E}^h_r(t)$ and $\mathcal{E}_r^l(t)$ effectively. Therefore, we have to introduce a more accurate estimate as follows.
\begin{lemma}\label{energy}
Let $d\geq2, a=0$ and $(u,\tau)$ be the solution of system \eqref{OBA} on $[0,T)$.
The following estimate holds
\begin{equation*}
\begin{split}
&\mathcal{E}^l(t)+\mathcal{E}^h(t)\leq C_1\mathcal{E}_0\\
&+
C_2\int_0^t \sum_{j\leq 0}2^{j(\frac{d}{2}-1)}\Big( \tilde{E}_j/(\|\dot{\Delta}_j u\|_{L^2}+\|\dot{\Delta}_j\Lambda^{-1}\mathbb{P}\nabla\cdot\tau\|_{L^2})+\tilde{Q}_j/
(\|\dot{\Delta}_j u\|_{L^2}+\|\dot{\Delta}_j\tau\|_{L^2})
\Big)\mathrm{d}t'\\
&+C_2\int_0^t\sum_{j>0} 2^{j(\frac{d}{2}-1)}\Big(\tilde{F}_j/(\|\dot{\Delta}_j u\|_{L^2}+\|\dot{\Delta}_j\mathbb{P}\nabla\cdot\tau\|_{L^2})+\tilde{R}_j/\|\dot{\Delta}_j u\|_{L^2}
+2^{j}\tilde{M}_j/\|\dot{\Delta}_j \tau\|_{L^2}\Big)\mathrm{d}t',
\end{split}
\end{equation*}
where $C_1, C_2$ independent of $T$.
\end{lemma}
As pointed out in Lemma \ref{originalLemma}, here we also omit the definitions of $\tilde{E}_j, \tilde{F}_j, \tilde{Q}_j,  \tilde{R}_j, \tilde{M}_j$. Please refer to the proof of Proposition \ref{pioriestimate} for their definitions.

\begin{proof} Without loss of generality, we assume
$$\|\dot{\Delta}_j u\|_{L^2},\  \|\dot{\Delta}_j\Lambda^{-1}\mathbb{P}\nabla\cdot\tau\|_{L^2},\  \|\dot{\Delta}_j\mathbb{P}\nabla\cdot\tau\|_{L^2},\ \|\dot{\Delta}_j\tau\|_{L^2}\neq 0. $$
\vskip 0.1cm
\textit{\bf{Step 1: Estimate of $\mathcal{E}^l(t)$.}}\label{step3}
\vskip 0.1cm
In this step we will supply the estimate of $\sup_t\|\tau\|^l_{\dot{B}^{\frac{d}{2}-1}_{2,1}}$ to $\mathcal{E}^l_r(t)$.

First, let us notice some cancellations on linear terms. \begin{eqnarray}{\label{cancellation1}}
&&(\dot{\Delta}_j\nabla {p},\dot{\Delta}_ju)=0,\\ \label{cancellation2}
&&(\dot{\Delta}_j\nabla \cdot\tau,\dot{\Delta}_ju)+(\dot{\Delta}_jD(u),\dot{\Delta}_j\tau)=0.
\end{eqnarray}
A standard energy computation of system \eqref{OBA} yields
\begin{eqnarray}\label{AU}
&&\frac{1}{2}\frac{d}{dt}\|\dot{\Delta}_j u\|_{L^2}^2+\nu\|\Lambda\dot{\Delta}_ju\|_{L^2}^2=\mu_1(\dot{\Delta}_j\nabla\cdot\tau,\dot{\Delta}_j u)
-(\dot{\Delta}_j(u\cdot\nabla u),\dot{\Delta}_ju).\\ \label{AT}
&&\frac{1}{2}\frac{d}{dt}\|\dot{\Delta}_j\tau\|_{L^2}^2=\mu_2(\dot{\Delta}_jD(u),\dot{\Delta}_j\tau)-(\dot{\Delta}_j(u\cdot\nabla \tau),\dot{\Delta}_j\tau)-(\dot{\Delta}_jQ(\tau,\nabla u),\dot{\Delta}_j\tau)
\end{eqnarray}
Calculating $\mu_2\eqref{AU}+\mu_1\eqref{AT}$, and thanks to \eqref{cancellation2}, we finally get
\begin{equation*}
\frac{1}{2}\frac{d}{dt}\Big(\|\dot{\Delta}_j u\|_{L^2}+\|\dot{\Delta}_j\tau\|_{L^2}\Big)\leq
C_2 \tilde{Q}_j/(\|\dot{\Delta}_j u\|_{L^2}+\|\dot{\Delta}_j\tau\|_{L^2}).
\end{equation*}
By a straightforward computation in the low frequencies, the above inequality can be estimated as
\begin{equation}\label{SupplyL}
\begin{split}
\sup_t\|u\|^l_{\dot{B}^{\frac{d}{2}-1}_{2,1}}+&\sup_t\|\tau\|^l_{\dot{B}^{\frac{d}{2}-1}_{2,1}}
\leq C_1( \|u_0\|^l_{\dot{B}^{\frac{d}{2}-1}_{2,1}}+\|\tau_0\|^l_{\dot{B}^{\frac{d}{2}-1}_{2,1}})
\\
&+C_2\int_0^t \sum_{j\leq 0}2^{j(\frac{d}{2}-1)}\tilde{Q}_j/(\|\dot{\Delta}_j u\|_{L^2}+\|\dot{\Delta}_j\tau\|_{L^2})\mathrm{d}t'.
\end{split}
\end{equation}
Combining \eqref{SupplyL} and \eqref{e0l}, we obtain
\begin{equation}\label{jia1}
\begin{split}
\mathcal{E}^l(t)\leq&
C_1(\|u_0\|^l_{\dot{B}^{\frac{d}{2}-1}_{2,1}}+\|\tau_0\|^l_{\dot{B}
^{\frac{d}{2}-1,\frac{d}{2}}})+ C_2\int_0^t \sum_{j\leq
0}2^{j(\frac{d}{2}-1)}\\&\times\Big( \tilde{E}_j/(\|\dot{\Delta}_j
u\|_{L^2}+\|\dot{\Delta}_j\Lambda^{-1}\mathbb{P}\nabla\cdot\tau\|_{L^2})+\tilde{Q}_j/
(\|\dot{\Delta}_j u\|_{L^2}+\|\dot{\Delta}_j\tau\|_{L^2})
\Big)\mathrm{d}t'.
\end{split}
\end{equation}
\vskip 0.1cm
\textit{\bf{Step 2: Estimate of $\mathcal{E}^h(t)$.}}\label{step4}
\vskip 0.1cm

In this step we will supply the estimate of $\sup_t\|\tau\|^h_{\dot{B}^{\frac{d}{2}}_{2,1}}$ to $\mathcal{E}^h_r(t)$.

It follows from the equality \eqref{AT} that
\begin{equation*}
\frac{d}{dt}\big(2^j\|\dot{\Delta}_j\tau\|_2\big)\leq C2^{2j}\|\dot{\Delta}_ju\|_{2}+C2^j\tilde{M}_j/\|\dot{\Delta}_j\tau\|_{L^2}.
\end{equation*}
The standard computation yields
\begin{equation}\label{Tauh}
\begin{split}
\sup_t\|\tau\|^h_{\dot{B}^{\frac{d}{2}}_{2,1}}\leq \|\tau_0\|^h_{\dot{B}^{\frac{d}{2}}_{2,1}}+ C\int_0^t\|u\|^h_{\dot{B}^{\frac{d}{2}+1}_{2,1}}\mathrm{d}t'
+C\int_0^t\sum_{j>0}2^{j\frac{d}{2}}\tilde{M}_j/\|\dot{\Delta}_j\tau\|_{L^2}\mathrm{d}t'.
\end{split}
\end{equation}
Calculating $\eqref{uhh}+\eta_2\eqref{Tauh}$, and if we choose $\eta_2$ small enough such that $\eta_2C\leq\frac{1}{2}$, then the term $\eta_2C\int_0^t\|u\|^h_{\dot{B}^{\frac{d}{2}+1}_{2,1}}\mathrm{d}t'$ can be absorbed by the left side, we eventually obtain
\begin{equation}\label{energyfinal}
\begin{split}
\mathcal{E}^h(t)\leq& C_1 (\|u_0\|^h_{\dot{B}^{\frac{d}{2}-1}_{2,1}}+\|\tau_0\|^h
_{\dot{B}^{\frac{d}{2}}_{2,1}})
+C_2\int_0^t\sum_{j\leq0}2^{j(\frac{d}{2}-1)}\Big(\tilde{F}_j/(\|\dot{\Delta}_j u\|_{L^2}\\&+\|\dot{\Delta}_j\mathbb{P}\nabla\cdot\tau\|_{L^2})
+\tilde{R}_j/\|\dot{\Delta}_j u\|_{L^2}
+2^{j}\tilde{M}_j/\|\dot{\Delta}_j \tau\|_{L^2}\Big)\mathrm{d}t'.
\end{split}
\end{equation}

Combining \eqref{jia1} and \eqref{energyfinal}, we complete the proof of Lemma \ref{energy}.
\end{proof}

Now let us turn to the proof of Proposition \ref{pioriestimate}.
\begin{proof}
We will apply Lemma \ref{energy} to prove \eqref{DD}. Let us give the definitions of the nonlinear terms in Lemma \ref{energy},
the term $\tilde{E}_j=\tilde{E}^1_j+\tilde{E}^2_j+\tilde{E}^3_j$ with
\begin{equation*}
\begin{split}
&\begin{split}
\tilde{E}_j^1=&-\frac{\mu_2}{2}(\dot{\Delta}_j\mathbb{P}(u\cdot\nabla u),\dot{\Delta}_j u)-\mu_1(\dot{\Delta}_j
\Lambda^{-1}\mathbb{P}\nabla\cdot Q(\tau,\nabla u),\dot{\Delta}_j\Lambda^{-1}\mathbb{P}\nabla\cdot\tau)  \\
&+K_1(\dot{\Delta}_j\mathbb{P}\nabla\cdot Q(\tau,\nabla u),\dot{\Delta}_j u),
\end{split}\\
&\tilde{E}_j^2=-\mu_1(\dot{\Delta}_j\Lambda^{-1}\mathbb{P}\nabla\cdot(u\cdot\nabla\tau),
\dot{\Delta}_j\Lambda^{-1}\mathbb{P}\nabla\cdot\tau), \\
&\tilde{E}_j^3=K_1(\dot{\Delta}_j\mathbb{P}(u\cdot\nabla u),\dot{\Delta}_j\mathbb{P}\nabla\cdot\tau)+K_1(\dot{\Delta}_j\mathbb{P}
\nabla\cdot(u\cdot\nabla\tau),\dot{\Delta}_ju),
\end{split}
\end{equation*}
and the term $\tilde{F}_j=\tilde{F}^1_j+\tilde{F}^2_j+\tilde{F}^3_j$ with
\begin{equation*}
\begin{split}
&
\begin{split}
\tilde{F}_j^1=&-\nu(\dot{\Delta}_j\mathbb{P}\nabla\cdot Q(\tau,\nabla u),\dot{\Delta}_j\mathbb{P}\nabla\cdot\tau)+\frac{\mu_2}{2}
(\dot{\Delta}_j\mathbb{P}\nabla\cdot Q(\tau,\nabla u),\dot{\Delta}_ju)\\
&-\frac{\mu_2^2}{\nu}(\dot{\Delta}_j\mathbb{P}(u\cdot\nabla u),\dot{\Delta}_ju),
\end{split}
\\
&\tilde{F}_j^2=-\nu(\dot{\Delta}_j\mathbb{P}\nabla\cdot(u\cdot\nabla \tau),\dot{\Delta}_j\mathbb{P}\nabla\cdot\tau),\\
&\tilde{F}_j^3=\frac{\mu_2}{2}(\dot{\Delta}_j\mathbb{P}(u\cdot\nabla u),\dot{\Delta}_j\mathbb{P}\nabla\cdot\tau)
+\frac{\mu_2}{2}(\dot{\Delta}_j\mathbb{P}\nabla\cdot(u\cdot\nabla\tau),\dot{\Delta}_ju),
\end{split}
\end{equation*}
and
\begin{equation*}
\begin{split}
&\tilde{Q}_j=-\mu_2(\dot{\Delta}_j(u\cdot\nabla u),\dot{\Delta}_ju)-\mu_1(\dot{\Delta}_j(u\cdot\nabla \tau),\dot{\Delta}_j\tau)-\mu_1(\dot{\Delta}_jQ(\tau,\nabla u),\dot{\Delta}_j\tau),\\
&\tilde{R}_j=-(\dot{\Delta}_j\mathbb{P}(u\cdot\nabla u),\dot{\Delta}_j u),\\
&\tilde{M}_j=-(\dot{\Delta}_j(u\cdot\nabla\tau),\dot{\Delta}_j\tau)-(\dot{\Delta}_j Q(\tau,\nabla u),\dot{\Delta}_j \tau),
\end{split}
\end{equation*}
\noindent for the constant $K_1$ is chosen in \eqref{UTL} in Lemma \ref{originalLemma}.
\vskip 0.1cm
\textit{\bf{Step 1: Estimate for $\tilde{E}_j^1,\ \tilde{F}_j^1,\ \tilde{Q}_j,\ \tilde{M}_j,\ \tilde{R}_j$.}}\label{step5}
\vskip 0.1cm

First we consider the terms of $\tilde{E}_j^1,\ \tilde{F}_j^1,\ \tilde{Q}_j,\ \tilde{M}_j,\ \tilde{R}_j$ excluding $Q(\tau,\nabla u)$. Noting that $\nabla\cdot u=0$, we have
\begin{equation}\label{BB}
(\dot{\Delta}_j\mathbb{P}(u\cdot\nabla u),\dot{\Delta}_j u)=(\dot{\Delta}_j(u\cdot\nabla u),\dot{\Delta}_j u).
\end{equation}
According to Proposition \ref{pro1} and Proposition \ref{pro2} in Appendix, we have
\begin{equation}\begin{aligned}
&|(\dot{\Delta}_j(u\cdot\nabla u),\dot{\Delta}_j u)|
\lesssim c_j2^{-j(\frac{d}{2}-1)}\|u\|_{\dot{B}^{\frac{d}{2}+1}_{2,1}}\|u\|
_{\dot{B}^{\frac{d}{2}-1}_{2,1}}\|\dot{\Delta}_ju\|_{L^2},\\
&|(\dot{\Delta}_j(u\cdot\nabla \tau),\dot{\Delta}_j \tau)| \lesssim
c_j2^{-j\psi^{\frac{d}{2}-1,\frac{d}{2}}(j)}\|u\|_{\dot{B}^{\frac{d}{2}+1}_{2,1}}
\|\tau\|_{\dot{B}^{\frac{d}{2}-1,\frac{d}{2}}}\|\dot{\Delta}_j\tau\|_{L^2},
\end{aligned}\end{equation}
with $\sum_{j\in\mathbb{Z}}c_j\leq1$.
Next we consider the terms of $\tilde{E}_j^1,\ \tilde{F}_j^1,\ \tilde{Q}_j,\ \tilde{R}_j,\ \tilde{M}_j$  including $Q(\tau,\nabla u)$. Applying Lemma \ref{benstein} and H\"{o}lder's inequality, for $j\leq 0$, we have
\begin{equation}\begin{aligned}
&\big|-\mu_1(\dot{\Delta}_j\Lambda^{-1}\mathbb{P}\nabla\cdot Q(\tau,\nabla u),\dot{\Delta}_j\Lambda^{-1}\mathbb{P}\nabla\cdot\tau)+K_1(\dot{\Delta}_j\mathbb{P}\nabla\cdot Q(\tau,\nabla u),\dot{\Delta}_j u)\big|\\&\lesssim(1+2^j)\| Q(\tau,\nabla u)\|_{L^2}(\|\dot{\Delta}_j u\|_{L^2}+\|\dot{\Delta}_j\Lambda^{-1}\mathbb{P}\nabla\cdot\tau\|_{L^2})\\
&\lesssim\| Q(\tau,\nabla u)\|_{L^2}(\|\dot{\Delta}_j u\|_{L^2}+\|\dot{\Delta}_j\Lambda^{-1}\mathbb{P}\nabla\cdot\tau\|_{L^2}),
\\
&\big|-\mu_1(\dot{\Delta}_jQ(\tau,\nabla u),\dot{\Delta}_j\tau)\big|\lesssim\| Q(\tau,\nabla u)\|_{L^2}\|\dot{\Delta}_j \tau\|_{L^2},
\end{aligned}\end{equation}
and $j\ge0$, we have
\begin{equation}\begin{aligned}\label{CC}
&\big|-\nu(\dot{\Delta}_j\mathbb{P}\nabla\cdot Q(\tau,\nabla u),\dot{\Delta}_j\mathbb{P}\nabla\cdot\tau)+\frac{\mu_2}{2}
(\dot{\Delta}_j\mathbb{P}\nabla\cdot Q(\tau,\nabla u),\dot{\Delta}_ju) \big| \\
&\lesssim2^j\| Q(\tau,\nabla u)\|_{L^2}(\|\dot{\Delta}_j u\|_{L^2}+\|\dot{\Delta}_j\mathbb{P}\nabla\cdot\tau\|_{L^2}),
\\
&\big|-(\dot{\Delta}_j Q(\tau,\nabla u),\dot{\Delta}_j \tau) \big|\lesssim\| Q(\tau,\nabla u)\|_{L^2}(\|\dot{\Delta}_j \tau\|_{L^2}+\|\dot{\Delta}_j\mathbb{P}\nabla\cdot\tau\|_{L^2}).
\end{aligned}\end{equation}
Combining $\eqref{BB}\sim\eqref{CC}$, we finally obtain
\begin{equation}\label{sum1}
\begin{split}
&\int_0^t\sum_{j\leq0}2^{j(\frac{d}{2}-1)}\Big(\tilde{E}_j^1/(\|\dot{\Delta}_ju\|_2\!+\!
\|\dot{\Delta}_j\Lambda^{-1}\mathbb{P}\nabla\cdot\tau\|_{L^2})\!+\! \tilde{Q}_j/(\|\dot{\Delta}_j u\|_{L^2}\!+\!\|\dot{\Delta}_j\tau\|_{L^2}\Big)\mathrm{d}t'\\
&\!\!+\!\!\int_0^t\!\!\sum_{j\leq0}2^{j(\frac{d}{2}-1)}\Big(\tilde{F}_j^1
/(\|\dot{\Delta}_ju\|_{L^2}\!+\!\|\dot{\Delta}_j
\mathbb{P}\nabla\!\cdot\!\tau\|_{L^2})\!+\!\tilde{R}_j/\|\dot{\Delta}_ju\|_{L^2}\!+\!2^j\tilde{M}_j/\|\dot{
\Delta}_j\tau\|_{L^2}\Big)\mathrm{d}t'\\
&\lesssim\big(\sup_t\|\tau\|_{\dot{B}^{\frac{d}{2}-1,\frac{d}{2}}}
+\sup_t\|u\|_{\dot{B}^{\frac{d}{2}-1}_{2,1}}\big)\int_0^t
\|u\|_{\dot{B}^{\frac{d}{2}+1}_{2,1}}\mathrm{d}t'+\int_0^t\|Q(\tau,\nabla u)\|_{\dot{B}^{\frac{d}{2}-1,\frac{d}{2}}}\mathrm{d}t'\\
&\lesssim \big(\sup_t\|\tau\|_{\dot{B}^{\frac{d}{2}-1,\frac{d}{2}}}+
\sup_t\|u\|_{\dot{B}^{\frac{d}{2}-1}_{2,1}}\big)\int_0^t
\|u\|_{\dot{B}^{\frac{d}{2}+1}_{2,1}}\mathrm{d}t',
\end{split}
\end{equation}
where in the second inequality we use Remark \ref{BY}.
\vskip 0.1cm
\textit{\bf{Step 2: Estimate for $\tilde{E}_j^2, \tilde{E}_j^3$.}}\label{step6}
\vskip 0.1cm
Next we turn to the term $\tilde{E}_j^2$. Our strategy is to apply Proposition \ref{Profile} and divide $\tilde{E}_j^2$ into three parts. We have
$\tilde{E}_j^2= \tilde{E}_j^{2,1}+\tilde{E}_j^{2,2}+\tilde{E}_j^{2,3}$, where
\begin{equation*}
\begin{split}
&\tilde{E}_j^{2,1}=-\mu_1(\dot{\Delta}_j\Lambda^{-1}\mathbb{P}(u\cdot\nabla\mathbb{P}\nabla\cdot\tau)
,\dot{\Delta}_j\Lambda^{-1}\mathbb{P}\nabla\cdot\tau),\\
&\tilde{E}_j^{2,2}=-\mu_1(\dot{\Delta}_j\Lambda^{-1}\mathbb{P}(\nabla u\cdot\nabla\tau),\dot{\Delta}_j\Lambda^{-1}\mathbb{P}\nabla\cdot\tau),\\
&\tilde{E}_j^{2,3}=\mu_1(\dot{\Delta}_j\Lambda^{-1}\mathbb{P}(\nabla u\cdot\nabla\Delta^{-1}\nabla\cdot\nabla\cdot\tau),\dot{\Delta}_j\Lambda^{-1}
\mathbb{P}\nabla\cdot\tau).
\end{split}
\end{equation*}
Let us consider the most difficult term $\tilde{E}_j^{2,1}$. We claim that
\begin{equation}\label{Ej2}
\|u\otimes\mathbb{P}\nabla\cdot\tau\|^l_{\dot{B}^{\frac{d}{2}-1}_{2,1}}\lesssim\|u\|_{\dot{B}^{
\frac{d}{2}-1}_{2,1}}\|\mathbb{P}\nabla\cdot\tau\|_{B^{\frac{d}{2}
,\frac{d}{2}-1}}+\|u\|_{\dot{B}^{\frac{d}{2}+1}_{2,1}}
\|\mathbb{P}\nabla\cdot\tau\|_{B^{\frac{d}{2}-2,\frac{d}{2}-1}}.
\end{equation}
Then by H\"{o}lder's inequality, $\nabla\cdot u=0$ and \eqref{Ej2}, we obtain
\begin{equation}\label{Ej21}
\begin{split}
&\int_0^t\sum_{j\leq0}2^{j(\frac{d}{2}-1)}|\tilde{E}_j^{2,1}|/(\|\dot{\Delta}_ju\|_{L^2}+
\|\dot{\Delta}_j\Lambda^{-1}\mathbb{P}\nabla\cdot\tau\|_{L^2})\mathrm{d}t'  \\
&\lesssim\int_0^t\|u\otimes\mathbb{P}\nabla\cdot\tau\|^l_{
\dot{B}^{\frac{d}{2}-1}_{2,1}}\mathrm{d}t' \\
&\lesssim\sup_t\|\tau\|_{\dot{B}^{\frac{d}{2}-1,\frac{d}{2}}}\int_
0^t\|u\|_{\dot{B}^{\frac{d}{2}+1}_{2,1}}\mathrm{d}t'
+\sup_t\|u\|_{\dot{B}^{\frac{d}{2}-1}_{2,1}}\int_0^t\|\mathbb{P}\nabla\cdot\tau\|
_{B^{\frac{d}{2},\frac{d}{2}-1}}\mathrm{d}t'.
\end{split}
\end{equation}
The proof of \eqref{Ej2} is
based on the Bony's product decomposition $$u\otimes\mathbb{P}\nabla\cdot\tau=\dot{T}_{u}\mathbb{P}\nabla\cdot\tau+\dot{T}_{\mathbb{P}
\nabla\cdot\tau}u
+\dot{R}(u,\mathbb{P}\nabla\cdot\tau).$$
where the definitions of $\dot{T}$ and $\dot{R}$ can be referred to Appendix.
Then by Proposition \ref{pro1}, we obtain
\begin{eqnarray*}
&&\|T_{u}\mathbb{P}\nabla\cdot\tau+T_{\mathbb{P}
\nabla\cdot\tau}u\|^l_{\dot{B}^{\frac{d}{2}-1}_{2,1}}
\lesssim\|T_{u}\mathbb{P}\nabla\cdot\tau+T_{\mathbb{P}
\nabla\cdot\tau}u\|_{\dot{B}^{\frac{d}{2}-1,\frac{d}{2}-2}}
\lesssim\|u\|_{\dot{B}^{\frac{d}{2}-1}_{2,1}}\|\mathbb{P}\nabla\cdot
\tau\|_{B^{\frac{d}{2},\frac{d}{2}-1}},\\
&&\|R(\mathbb{P}\nabla\cdot\tau,u)\|^l_{\dot{B}^{\frac{d}{2}-1}_{2,1}}
\lesssim \|R(\mathbb{P}\nabla\cdot\tau,u)\|_{\dot{B}^{\frac{d}{2}-1,\frac{d}{2}
}}\lesssim\|u\|_{\dot{B}^{\frac{d}{2}+1}_{2,1}}\|\mathbb{P}\nabla\cdot\tau\|
_{B^{\frac{d}{2}-2,\frac{d}{2}-1}},
\end{eqnarray*}
which imply \eqref{Ej2}.
Let us return to the terms $\tilde{E}_j^{2,2},\tilde{E}_j^{2,3}$. Noticing that the incompressible condition on $u$, we have the following equalities:
\begin{equation}\label{fenbujifen1}
\begin{split}
&[\nabla u\cdot\nabla\tau]^i\triangleq\sum_{j,k}\partial_ju^k\partial_k\tau^{i,j}=\sum_{j,k}
\partial_k(\partial_ju^k\tau^{i,j}),\\
&[\nabla u\cdot\nabla\Delta^{-1}\nabla\cdot\nabla\cdot\tau]^i\triangleq\sum_{k}\partial_iu^k
\partial_k\Delta^{-1}\nabla\cdot\nabla\cdot\tau
=\sum_{k}
\partial_k(\partial_iu^k
\nabla\Delta^{-1}\nabla\cdot\nabla\cdot\tau).
\end{split}
\end{equation}
Combining \eqref{fenbujifen1} and Remark \ref{BY}, we have
\begin{equation}\label{Ej22Ej23}
\begin{split}
&\int_0^t\sum_{j\leq0}2^{j(\frac{d}{2}-1)}\big|\tilde{E}_j^{2,2}+\tilde{E}_j^{2,3}\big|/
(\|\dot{\Delta}_ju\|_{L^2}+
\|\dot{\Delta}_j\Lambda^{-1}\mathbb{P}\nabla\cdot\tau\|_{L^2})\mathrm{d}t' \\
&\lesssim\int_0^t\|\nabla u\cdot\nabla\tau\|^l_{\dot{B}^{\frac{d}{2}-2}_{2,1}}+\|\nabla u\cdot\nabla\Delta^{-1}\nabla\cdot\nabla\cdot\tau\|^l_{\dot{B}^{\frac{d}{2}-2}_{2,1}}
\mathrm{d}t'\\
&\lesssim\int_0^t\|\nabla u\otimes\tau\|_{\dot{B}^{\frac{d}{2}-1,\frac{d}{2}}}+\|\nabla u\otimes\Delta^{-1}\nabla\cdot\nabla\cdot\tau\|_{\dot{B}^{\frac{d}{2}-1,\frac{d}{2}}}
\mathrm{d}t'\\
& \lesssim \sup_t\|\tau\|_{\dot{B}^{\frac{d}{2}-1,\frac{d}{2}}}\int_0^t\|u\|_
{\dot{B}^{\frac{d}{2}+1}_{2,1}}\mathrm{d}t'.
\end{split}
\end{equation}
Therefore, it follows from \eqref{Ej21} and \eqref{Ej22Ej23} that
\begin{equation}\label{sum2}
\int_0^t\sum_{j\leq0}2^{j(\frac{d}{2}-1)}|\tilde{E}_j^{2}|/(\|\dot{\Delta}_ju\|_{L^2}+
\|\dot{\Delta}_j\Lambda^{-1}\mathbb{P}\nabla\cdot\tau\|_{L^2})\mathrm{d}t' \lesssim \mathcal{E}^2(t).
\end{equation}

Now we consider $\tilde{E}_j^3$.  Applying  Proposition \ref{Profile}, we have $\tilde{E}_j^3=\tilde{E}_j^{3,1}+\tilde{E}_j^{3,2}+\tilde{E}_j^{3,3}$, where
\begin{equation*}
\begin{split}
&\tilde{E}_j^{3,1}=K_1\big((\dot{\Delta}_j\mathbb{P}(u\cdot\nabla
 u),\dot{\Delta}_j\mathbb{P}\nabla\cdot\tau)+(\dot{\Delta}_j\mathbb{P}
(u\cdot\nabla\mathbb{P}\nabla\cdot\tau),\dot{\Delta}_ju)\big), \\
&\tilde{E}_j^{3,2}=K_1(\dot{\Delta}_j\mathbb{P}(\nabla u\cdot\nabla\tau),\dot{\Delta}_ju),\\
&\tilde{E}_j^{3,3}=-K_1(\dot{\Delta}_j\mathbb{P}(\nabla u\cdot\nabla\Delta^{-1}\nabla\cdot\nabla\cdot\tau),\dot{\Delta}_ju).
\end{split}
\end{equation*}
For the term $\tilde{E}_j^{3,1}$, noticing that $j\leq0$ and  Proposition \ref{pro2}, we obtain
\begin{equation*}
\begin{split}
|\tilde{E}_j^{3,1}| \lesssim
c_j\|u\|_{\dot{B}^{\frac{d}{2}+1}_{2,1}}\Big(2^{-j(\frac{d}{2}-1)}\|u\|_{\dot{B}^{\frac{d}{2}-1}
_{2,1}}
\|\dot{\Delta}_j\mathbb{P}\nabla\cdot\tau\|_{L^2}+2^{-j(\frac{d}{2}-2)}\|\mathbb{P}
\nabla\cdot\tau\|_{\dot{B}^{\frac{d}{2}-2,\frac{d}{2}-1}}\|\dot{\Delta}_ju\|_{L^2}\Big).
\end{split}
\end{equation*}
Thanks to
$$\|\dot{\Delta}_j\mathbb{P}\nabla\cdot\tau\|_{L^2}\lesssim
2^j\|\dot{\Delta}_j\Lambda^{-1}\mathbb{P}\nabla\cdot\tau\|_{L^2},$$
and $j\le 0$, we have
\begin{equation}\label{Ej31}
\begin{split}
&\int_0^t\sum_{j\leq0}2^{j(\frac{d}{2}-1)}|\tilde{E}_j^{3,1}|/(\|\dot{\Delta}_ju\|_{L^2}+
\|\dot{\Delta}_j\Lambda^{-1}\mathbb{P}\nabla\cdot\tau\|_{L^2})\mathrm{d}t'  \\
 &\lesssim \big( \sup_t\|u\|_{\dot{B}^{\frac{d}{2}-1}_{2,1}} + \sup_t\|\tau\|_{\dot{B}^{\frac{d}{2}-1,\frac{d}{2}}}\big)\int_0^t\|u\|
_{\dot{B}^{\frac{d}{2}+1}_{2,1}}\mathrm{d}t'.
\end{split}
\end{equation}
Dealing with the terms $\tilde{E}_j^{3,2}, \tilde{E}_j^{3,3}$ in the same way as used in the proof of $\tilde{E}_j^{2,2}, \tilde{E}_j^{2,3}$, we have
\begin{equation}\label{Ej32Ej33}
\begin{split}
&\int_0^t\sum_{j\leq0}2^{j(\frac{d}{2}-1)}\big|\tilde{E}_j^{3,2}+\tilde{E}_j^{3,3}\big|/
(\|\dot{\Delta}_ju\|_{L^2}+
\|\dot{\Delta}_j\Lambda^{-1}\mathbb{P}\nabla\cdot\tau\|_{L^2})\mathrm{d}t' \\
& \lesssim \sup_t\|\tau\|_{\dot{B}^{\frac{d}{2}-1,\frac{d}{2}}}\int_0^t\|u\|_
{\dot{B}^{\frac{d}{2}+1}_{2,1}}\mathrm{d}t'.
\end{split}
\end{equation}

Then combining \eqref{Ej31} and \eqref{Ej32Ej33}, we deduce that
\begin{equation}\label{sum3}
\int_0^t\sum_{j\leq0}2^{j(\frac{d}{2}-1)}|\tilde{E}_j^{3}|/(\|\dot{\Delta}_ju\|_{L^2}+
\|\dot{\Delta}_j\Lambda^{-1}\mathbb{P}\nabla\cdot\tau\|_{L^2})\mathrm{d}t' \lesssim \mathcal{E}^2(t).
\end{equation}
\vskip 0.1cm
\textit{\bf{Step 3: Estimate for $\tilde{F}_j^2, \tilde{F}_j^3$.}}\label{step7}
\vskip 0.1cm
Using Proposition \ref{Profile}, we divide $\tilde{F}^2_j$ into three terms
\begin{equation*}
\begin{split}
&\tilde{F}_j^{2,1}=-\nu(\dot{\Delta}_j\mathbb{P}(u\cdot\nabla\mathbb{P}\nabla\cdot\tau),
\dot{\Delta}_j\mathbb{P}\nabla\cdot\tau),\\
&\tilde{F}_j^{2,2}=-\nu(\dot{\Delta}_j\mathbb{P}(\nabla u\cdot\nabla\tau),\dot{\Delta}_j\mathbb{P}\nabla\cdot\tau),\\
&\tilde{F}_j^{2,3}=\nu(\dot{\Delta}_j\mathbb{P}(\nabla u\cdot\nabla\Delta^{-1}\nabla\cdot\nabla\cdot\tau),\dot{\Delta}_j
\mathbb{P}\nabla\cdot\tau).
\end{split}
\end{equation*}
By Proposition \ref{pro2}, we obtain
\begin{equation*}
|(\dot{\Delta}_j(u\cdot\nabla\mathbb{P}\nabla\cdot\tau),\dot{\Delta}_j\mathbb{P}
\nabla\cdot\tau)|\lesssim c_j2^{-j(\frac{d}{2}-1)}\|u\|_{\dot{B}^{\frac{d}{2}+1}_{2,1}}
\|\mathbb{P}\nabla\cdot\tau\|_{\dot{B}^{\frac{d}{2}-2,\frac{d}{2}-1}}\|
\dot{\Delta}_j\mathbb{P}\nabla\cdot\tau\|_{L^2},
\end{equation*}
which implies
\begin{equation}\label{Fj21}
\begin{split}
&\int_0^t\sum_{j>0}2^{j(\frac{d}{2}-1)}|\tilde{F}_j^{2,1}|/(\|\dot{\Delta}_ju\|_{L^2}+
\|\dot{\Delta}_j\mathbb{P}\nabla\cdot\tau\|_{L^2})\mathrm{d}t'\\
 &\lesssim \sup_t\|\tau\|_{\dot{B}^{\frac{d}{2}-1,\frac{d}{2}}}\int_0^t\|u\|_
{\dot{B}^{\frac{d}{2}+1}_{2,1}}\mathrm{d}t'.
\end{split}
\end{equation}
Then we consider $\tilde{F}_j^{2,2}$ and $\tilde{F}_j^{2,3}$. Thanks to Remark \ref{BY} and \eqref{fenbujifen1}, we obtain
\begin{equation}\label{Fj22Fj23}
\begin{split}
&\int_0^t\sum_{j>0}2^{j(\frac{d}{2}-1)}\big|\tilde{F}^{2,2}_j+\tilde{F}^{2,3}_j\big|/(\|\dot{\Delta}_ju\|_{L^2}+
\|\dot{\Delta}_j\Lambda^{-1}\mathbb{P}\nabla\cdot\tau\|_{L^2})\mathrm{d}t'  \\
&\lesssim\int_0^t\|\nabla u\otimes\tau\|_{\dot{B}^{\frac{d}{2}-1,\frac{d}{2}}}+\|\nabla u\otimes\Delta^{-1}\nabla\cdot\nabla\cdot\tau\|_{\dot{B}^{\frac{d}{2}-1,\frac{d}{2}}}
\mathrm{d}t'\\
& \lesssim \sup_t\|\tau\|_{\dot{B}^{\frac{d}{2}-1,\frac{d}{2}}}\int_0^t\|u\|_
{\dot{B}^{\frac{d}{2}+1}_{2,1}}\mathrm{d}t'.
\end{split}
\end{equation}
Combining \eqref{Fj21} and \eqref{Fj22Fj23}, we obtain
\begin{equation}\label{sum4}
\int_0^t\sum_{j\leq0}2^{j(\frac{d}{2}-1)}|\tilde{F}_j^{2}|/(\|\dot{\Delta}_ju\|_{L^2}+
\|\dot{\Delta}_j\Lambda^{-1}\mathbb{P}\nabla\cdot\tau\|_{L^2})\mathrm{d}t' \lesssim \mathcal{E}^2(t).
\end{equation}

For the last term $\tilde{F}^3_j$, we also have $\tilde{F}^3_j=\tilde{F}^{3,1}_j+\tilde{F}_j^{3,2}+\tilde{F}^{3,3}_j$, where
\begin{equation*}
\begin{split}
&\tilde{F}^{3,1}_j=\frac{\mu_2}{2}\big((\dot{\Delta}_j(u\cdot\nabla u),\dot{\Delta}_j\mathbb{P}\nabla\cdot\tau)+(\dot{\Delta}_j(u\cdot\nabla
\mathbb{P}\nabla\cdot\tau),\dot{\Delta}_ju)\big),\\
&\tilde{F}^{3,2}_j=\frac{\mu_2}{2}(\dot{\Delta}_j\mathbb{P}(\nabla u\cdot\nabla\tau),\dot{\Delta}_ju),\\
&\tilde{F}^{3,3}_j=-\frac{\mu_2}{2}(\dot{\Delta}_j\mathbb{P}(\nabla u\cdot\nabla\Delta^{-1}\nabla\cdot\nabla\cdot\tau),\dot{\Delta}_ju).
\end{split}
\end{equation*}
Let us first consider $\tilde{F}_j^{3,1}$. By  Proposition \ref{pro2}, we obtain
\begin{equation*}
\begin{split}
&\big|(\dot{\Delta}_j(u\cdot\nabla u),\dot{\Delta}_j\mathbb{P}\nabla\cdot\tau)+
(\dot{\Delta}_j(u\cdot\nabla\mathbb{P}\nabla\cdot\tau),\dot{\Delta}_ju)\big|
\\&\lesssim c_j\|u\|_{\dot{B}^{\frac{d}{2}+1}_{2,1}}\big(2^{-j(\frac{d}{2}-1)}\|u\|_{\dot{B}^{\frac{d}{2}-1}_{2,1}}\|\mathbb{P}\nabla\cdot\tau\|_{L^2}
+2^{-j(\frac{d}{2}-1)}\|\mathbb{P}\nabla\cdot\tau\|_{\dot{B}^{\frac{d}{2}-2,
\frac{d}{2}-1}}\|\dot{\Delta}_ju\|_{L^2}\big).
\end{split}
\end{equation*}
Then, we have
\begin{equation}\label{ssum4}
\begin{split}
&\int_0^t\sum_{j>0}2^{j(\frac{d}{2}-1)}|\tilde{F}_j^{3,1}|/(\|\dot{\Delta}_ju\|_{L^2}+
\|\dot{\Delta}_j\Lambda^{-1}\mathbb{P}\nabla\cdot\tau\|_{L^2})\mathrm{d}t'\\
&\lesssim\big(\sup_t\|u\|_{\dot{B}^{\frac{d}{2}-1}_{2,1}}+
\sup_t\|\tau\|_{\dot{B}^{\frac{d}{2}-1,\frac{d}{2}}}\big)\int_0^t\!\|u\|_{\dot{B}
^{\frac{d}{2}+1}_{2,1}}\mathrm{d}t'.
\end{split}
\end{equation}
Dealing with the terms $\tilde{F}_j^{3,2}, \tilde{F}_j^{3,3}$ in the same way as used in the proof of $\tilde{F}_j^{2,2}, \tilde{F}_j^{2,3}$, we have
\begin{equation}\label{Fj32Fj33}
\begin{split}
&\int_0^t\sum_{j>0}2^{j(\frac{d}{2}-1)}\big|\tilde{F}^{2,2}_j+\tilde{F}^{2,3}_j\big|/(\|\dot{\Delta}_ju\|_{L^2}+
\|\dot{\Delta}_j\Lambda^{-1}\mathbb{P}\nabla\cdot\tau\|_{L^2})\mathrm{d}t'  \\
& \lesssim \sup_t\|\tau\|_{\dot{B}^{\frac{d}{2}-1,\frac{d}{2}}}\int_0^t\|u\|_
{\dot{B}^{\frac{d}{2}+1}_{2,1}}\mathrm{d}t'.
\end{split}
\end{equation}
Combining \eqref{ssum4}, \eqref{Fj32Fj33}, we gather
\begin{equation}\label{sum5}
\begin{split}
\int_0^t\sum_{j>0}2^{j(\frac{d}{2}-1)}|\tilde{F}_j^3|/(\|\dot{\Delta}_ju\|_{L^2}+
\|\dot{\Delta}_j\Lambda^{-1}\mathbb{P}\nabla\cdot\tau\|_{L^2})\mathrm{d}t' \lesssim \mathcal{E}^2(t).
\end{split}
\end{equation}
 According to \eqref{sum1}, \eqref{sum2}, \eqref{sum3}, \eqref{sum4} and \eqref{sum5}, we complete the proof of Proposition \ref{pioriestimate}.                                                        \end{proof}

\section{The global existence and the uniqueness}\label{Global}
\setcounter{section}{4}\setcounter{equation}{0}

This section is devoted to the proof of Theorem \ref{Maintheorem}.
\subsection{Approximate solutions and the uniform estimates}
The construction of approximate solutions is based on the following local existence theorem.
\begin{proposition}\label{pro3}
Let $(u_0,\tau_0)$ be an initial data in $H^s$ with $s$ strictly greater than $\frac{d}{2}$. Then a unique strictly positive time $T$ exists so that a unique solution $(u,\tau)$ exists such that
$$u\in C([0,T);H^s)\cap L^2_{loc}(0,T;H^{s+1});\ \tau\in C([0,T);H^s).$$
Moreover, the solution $(u,\tau)$ can be continued beyond $T$ if
$$\sup_T \|\tau\|_{\dot{B}^{\frac{d}{2}-1,\frac{d}{2}}}+\int_0^T\|u\|_{\dot{B}^{\frac{d}{2}+1}_{2,1}}\mathrm{d}t'<\infty.$$

\end{proposition}
\begin{proof}
The proof is very similar to the Theorem 1.1 in \cite{J.-Y.Chemin}, here we omit it since we only have made a slight modification of its proof.
\end{proof}
Set
$\mathcal{C}_n\triangleq\{\xi\in\mathbb{R}^d\,\ |\ \ n^{-1}\leq|\xi|\leq n\}$. We build the approximate solution $(u^{n},\tau^{n})$ solving the system
\begin{equation}\label{OBN}
\begin{cases}
u^{{n}}_t+u^{{n}}\cdot\nabla u^{{n}}-\nu\Delta u^{{n}}-\nabla p=\mu_1\nabla\cdot\tau^{{n}}\ \ \\
\tau^{{n}}_t+u^{{n}}\cdot\nabla\tau^{{n}}+Q(\tau^{{n}},\nabla u^{{n}})=\mu_2D(u^{{n}}),  \ \\
\nabla\cdot u^{{n}}=0, \\
u^n(0,x)=J_nu_0;\ \  \tau^{n}(0,x)=J_n\tau_0,
\end{cases}
\end{equation}
where $\mathcal{F}(J_nU)(\xi)=I_{\mathcal{C}_n}(\xi)\mathcal{F}U(\xi)$ with $I_{\mathcal{C}_n}$ the smooth cut-off functions supported in $\mathcal{C}_n$.
Using the direct computations, we gather that
\begin{equation}\label{lim}
\lim_{n\rightarrow \infty} \|J_n u_0-u_0\|_{\dot{B}^{\frac{d}{2}-1}_{2,1}}=0; \ \ \ \lim_{n\rightarrow \infty}\|J_n \tau_0-\tau_0\|_{\dot{B}^{\frac{d}{2}-1,\frac{d}{2}}}=0.
\end{equation}
Indeed, it is easy to see that $J_nu_0,\ J_n\tau_0\in H^s$ for all $s>0$. Therefore, applying  Proposition \ref{pro3}, we can obtain that there exists a maximal existence time $T_n>0$ such that for all $s>0$ the system \eqref{OBN} has a unique solution $(u^n,\tau^n)$,
\begin{equation*}
u^n\in C([0,T_n);H^s)\cap L^2_{loc}(0,T_n;H^{s+1});\ \ \ \tau^n\in C([0,T_n);H^s).
\end{equation*}
Using the definition of the Besov space, it is easy to check that
\begin{eqnarray*}
&&u^n\in C([0,T_n);\dot{B}^{\frac{d}{2}-1}_{2,1})\cap L^1_{loc}(0,T_n;\dot{B}^
{\frac{d}{2}+1}_{2,1});\\
&&\tau^n\in C([0,T_n);\dot{B}^{\frac{d}{2}-1,\frac{d}{2}}),\ \ \mathbb{P}\nabla\cdot\tau^n\in L^1_{loc}(0,T_n;\dot{B}^{\frac{d}{2},\frac{d}{2}-1}).
\end{eqnarray*}
Let us define
$$T_n^{*}=\sup \{ t\in[0,T_n)\ |\ \ \mathcal{E}^n(t)\leq\widetilde{C}\mathcal{E}_0\}.$$

Firstly, we claim that
$$T^{*}_n=T_n.$$
Using the continuity argument, it suffices to show that for all $n\in\mathbb{N}$,$$\mathcal{E}^n(t)\leq\frac{1}{2}\widetilde{C}\mathcal{E}_0.$$
In fact, set $\widetilde{C}=4C_1$, and choose $\mathcal{E}_0$ small enough such that $\mathcal{E}_0\leq\frac{1}{16C_1C_2}$, then combine Proposition \ref{pioriestimate}, we obtain
\begin{equation*}
\begin{split}
\mathcal{E}^n(t)\leq \big(C_1+16C_1^2C_2\mathcal{E}_0\big)\mathcal{E}_0
\leq \frac{1}{2}\widetilde{C}\mathcal{E}_0.
\end{split}
\end{equation*}
In conclusion, we construct a sequence of approximate solution $(u^n, \tau^n)$ on $[0,T_n)$ satisfying
\begin{equation}\label{ETNS}
\mathcal{E}^n(T_n)\leq4C_1\mathcal{E}_0,
\end{equation}
for any $n\in\mathbb{N}$.
Thanks to \eqref{ETNS}, we can easily obtain that $T_n=\infty$ by a direct application of Proposition \ref{pro3}. To sum up, we get that for all $t>0$, we have
\begin{equation}\label{ETN}
\mathcal{E}^n(t)\leq4C_1\mathcal{E}_0.
\end{equation}

\subsection{The existence}
In this part we will use a standard compact argument to  show that,
up to an extraction, the sequence $((u^n,\tau^n))_{n\in\mathbb{N}}$
converges in $\mathcal{D}'(\mathbb{R}^+\times\mathbb{R}^d)$ to a
solution $(u,\tau)$ of \eqref{OBA} which has the desired regularity
properties.

It is convenient to split $(u^n,\tau^n)$ into linear part and discrepant part. More precisely, we denote by $(u^n_L,\tau^n_L)$ the solution to
\begin{equation*}\label{}
\begin{cases}
\partial_t{u^n_L}-\nu\Delta u_L^{{n}}-\mu_1\nabla\cdot\tau^{{n}}_L-\nabla p=0,\ \ \\
\partial_t{\tau^n_L}-\mu_2D(u^n_L)=0,  \ \\
\nabla\cdot u^{{n}}_L=0, \\
u^n_L(0,x)=J_nu_0;\ \  \tau^{n}_L(0,x)=J_n\tau_0.
\end{cases}
\end{equation*}
By Proposition \ref{pioriestimate}, we can easily get $\mathcal{E}^n_L(t)\lesssim\mathcal{E}(0)$, where $\mathcal{E}^n_L(t)$ is the energy of $(u^n_L,\tau^n_L)$ in form of \eqref{energytotal}.
Also, we denote $(u_L,\tau_L)$ the solution to
\begin{equation*}\label{}
\begin{cases}
\partial_t{u_L}-\nu\Delta u_L-\mu_1\nabla\cdot\tau_L-\nabla p=0,\ \ \\
\partial_t{\tau_L}-\mu_2D(u_L)=0,  \ \\
\nabla\cdot u_L=0, \\
u_L(0)=u_0;\ \  \tau_L(0)=\tau_0.
\end{cases}
\end{equation*}
It is easy to check that
\begin{equation}\label{uL}\begin{aligned}
&u^n_L\ \longrightarrow \ u_L \quad \text{in}\quad
C(\mathbb{R}^+;\dot{B}^{\frac{d}{2}-1}_{2,1})\cap
L^1(\mathbb{R}^+;\dot{B}^{\frac{d}{2}+1}_{2,1}),\\ &\tau^n_L\
\longrightarrow\ \tau_L\quad  \text{in}\quad
C(\mathbb{R}^+;\dot{B}^{\frac{d}{2}-1,\frac{d}{2}});\quad\mathbb{P}\nabla\cdot\tau^n_L\
\longrightarrow\ \mathbb{P}\nabla\cdot\tau_L \quad \text{in}\quad
L^1(\mathbb{R}^+;\dot{B}^{\frac{d}{2},\frac{d}{2}-1}).
\end{aligned}\end{equation}
Denote $(\overline{u}^n,\overline{\tau}^n)\triangleq(u^n-u^n_L,\tau^n-\tau^n_L)$, thanks to \eqref{ETN}, we also get
\begin{eqnarray*}
&&\overline{u}^n\in C(\RR^+;\dot{B}^{\frac{d}{2}-1}_{2,1})\cap L^1(\RR^+;\dot{B}^{\frac{d}{2}+1}_{2,1}),\\
&&\overline{\tau}^n\in C(\RR^+;\dot{B}^{\frac{d}{2}-1,\frac{d}{2}});\ \ \
\mathbb{P}\nabla\cdot\overline{\tau}^n\in L^1(\RR^+;\dot{B}^{\frac{d}{2},\frac{d}{2}-1}).
\end{eqnarray*}
\begin{lemma}\label{LC}
$((\overline{u}^n,\overline{\tau}^n))_{n\in\mathbb{N}}$ is uniformly bounded in $
C^{\frac{1}{2}}_{loc}(\mathbb{R}^+;\dot{B}^{\frac{d}{2}-2}_{2,1})\times C^{\frac{1}{2}}_{loc}(\mathbb{R}^+;\dot{B}^{\frac{d}{2}-1}_{2,1})$.
\end{lemma}
\begin{proof}
Recall that
$$\partial_t\overline{\tau}^{{n}}=\mu_2D(\overline{u}^{{n}})-u^{{n}}\cdot\nabla\tau^{{n}}-Q(\tau^{{n}},\nabla u^{{n}}),$$
which combining with \eqref{ETN} and Proposition \ref{pro1}, we gather that
$$\|\partial_t\overline{\tau}^n\|_{L_T^2(\dot{B}^{\frac{d}{2}-1}_{2,1})}
\lesssim\|\overline{u}^n\|_{L^2_T(\dot{B}^{\frac{d}{2}}_{2,1})}+
\sup_T\|\tau^n\|_{\dot{B}^{\frac{d}{2}}_{2,1}}\|u^n\|_{L^2_T(\dot{B}^{\frac{d}{2}}_{2,1})}<\infty.$$
Then we get $\overline\tau^n\in C^{\frac{1}{2}}_{loc}(\mathbb{R}^+;\dot{B}^{\frac{d}{2}-1}_{2,1})$. On the other hand,
$$\partial_t\overline{u}^n=\mu_1\mathbb{P}\nabla\cdot\overline{\tau}^{{n}}+\nu\Delta \overline{u}^n+\mathbb{P}u^{{n}}\cdot\nabla u^{{n}}.$$
Proposition \ref{pro1} and \eqref{ETN} imply that
$$\|\partial_t\overline{u}^n\|_{L_T^2(\dot{B}^{\frac{d}{2}-2}_{2,1})}
\lesssim T^{\frac{1}{2}}\sup_T\|\overline{\tau}^n\|_{\dot{B}^{\frac{d}{2}-1,\frac{d}{2}}}+
(\sup_T\|u^n\|_{\dot{B}^{\frac{d}{2}-1}_{2,1}}+1)\|u^n\|_{L^2_T(\dot{B}^{\frac{d}
{2}}_{2,1})}<\infty,$$
which means that $\overline{u}^n\in C^{\frac{1}{2}}_{loc}(\mathbb{R}^+;\dot{B}^{\frac{d}{2}-2}_{2,1})$.
\end{proof}

Let us choose a sequence $(\phi_p)_{p\in\mathbb{N}}$ of smooth
cut-off functions supported in the ball $B(0,p+1)$ of $\mathbb{R}^d$
and equal to $1$ in a neighborhood of $B(0,p)$. Lemma \ref{LC}
ensures that
$$(\phi_p\overline{u}^n,\phi_p\overline{\tau}^n) \ \ \text{is uniformly equicontinuous in} \ \ C([0,p];\dot{B}^{\frac{d}{2}-2}_{2,1}\times\dot{B}^{\frac{d}{2}-1}_{2,1}).$$
We know that when $s>0$, $\dot{B}^s_{2,1}(K)\cong B^s_{2,1}(K)$ for
all compact set $K$ and $\dot{B}^s_{2,1}\hookrightarrow B^s_{2,1}$
when $s\leq0$. Then we can get
\begin{eqnarray}\label{equicontin}
&&(\phi_p\overline{u}^n,\phi_p\overline{\tau}^n) \ \ \text{is
uniformly equicontinuous in} \ \ C([0,p];
B^{\frac{d}{2}-2}_{2,1}\times
B^{\frac{d}{2}-1}_{2,1}),\\\label{unifromofn}
&&(\phi_p\overline{u}^n,\phi_p\overline{\tau}^n)\ \ \text{is
uniformly bounded in} \ \ \ C([0,p];B^{\frac{d}{2}-1}_{2,1}\times
B^{\frac{d}{2}}_{2,1}).
\end{eqnarray}
Moreover, we have the facts
$B^{\frac{d}{2}-1}_{2,1}(K)\hookrightarrow\hookrightarrow
B^{\frac{d}{2}-2}_{2,1}(K)$ and
$B^{\frac{d}{2}}_{2,1}(K)\hookrightarrow\hookrightarrow
B^{\frac{d}{2}-1}_{2,1}(K)$. Then, by Ascoli's Theorem and Cantor's
diagonal process, we get a distribution
$(\overline{u},\overline{\tau})$ such that
\begin{equation}\label{SL}
(\phi_p\overline{u}^n,\phi_p\overline{\tau}^n)
\longrightarrow(\phi_p\overline{u},\phi_p\overline{\tau})\ \ \ \text{in}\
 \ \ C([0,p]; B^{\frac{d}{2}-2}_{2,1}\times B^{\frac{d}{2}-1}_{2,1}).
\end{equation}
Denote $u\triangleq\overline{u}+u_L,\
\tau\triangleq\overline{\tau}+\tau_L$, we easily have
\begin{equation*}\label{}
(u^n,\tau^n)\longrightarrow(u,\tau)\ \ \ \text{in}\ \ \
\mathcal{S}'(\mathbb{R}^d\times\mathbb{R}^+),
\end{equation*}
With the help of \eqref{uL}, \eqref{unifromofn} and \eqref{SL},
following the argument as in \cite{R.DANCHIN2000}, it is routine
to verify that $(u,\tau)$ satisfies the system \eqref{OBA} in the
distribution sense. Moreover, we can infer from \eqref{ETN} that
\begin{equation} \label{UTSPACE}\begin{aligned}
&u\in L^\infty(\RR^+;\dot{B}^{\frac{d}{2}-1}_{2,1})\cap
L^1(\RR^+;\dot{B}^{\frac{d}{2}+1}_{2,1});
\\&\tau\in
L^\infty(\RR^+;\dot{B}^{\frac{d}{2}-1,\frac{d}{2}}),\,\,\mathbb{P}\nabla\cdot\tau\in
L^1(\RR^+;\dot{B}^{\frac{d}{2},\frac{d}{2}-1}).
\end{aligned}\end{equation}

At last, we have to show  the properties of continuity with respect
to time. The continuity of $u$ is straightforward. Indeed from the
equation of $u$, we have $
L^1_{loc}(\mathbb{R}^+;\dot{B}^{\frac{d}{2}-1}_{2,1})$ which imply
$u\in C(\RR^+;\dot{B}^{\frac{d}{2}-1}_{2,1}).$  As for $\tau$, by
the same argument as  in \cite{R.DANCHIN2000}, we get $\tau\in
C(\mathbb{R}^+;\dot{B}^{\frac{d}{2}-1,\frac{d}{2}})$.

\subsection{The uniqueness}\label{Uniqueness}
Assume $(u_1,\tau_1)$ and $(u_2,\tau_2)$ are two solutions of
\eqref{OBA} with the same initial data. Denote $\delta
u=u_1-u_2,\delta\tau=\tau_1-\tau_2$, then $(\delta u, \delta \tau)$
satisfies
\begin{equation}\label{cha}
\begin{cases}
(\delta u)_t-\nu \Delta\delta u=\mu_1\nabla\cdot \delta\tau -u_1\cdot\nabla \delta u-\delta u\cdot\nabla u_2,\\
(\delta\tau)_t+u_1\cdot\nabla\delta\tau=\mu_2D(\delta u)-\delta u \cdot\nabla\tau_2-Q(\tau_1,\nabla\delta u)-Q(\delta\tau,\nabla u_2),\\
\nabla\cdot\delta u=0,\\
\delta u(0,x)=0;\ \delta\tau(0,x)=0.
\end{cases}
\end{equation}
We shall work  in the following functional spaces
$$\mathcal{H}_T\triangleq\big(L^\infty(0,T;\dot{B}^{-1}_{d,\infty})\cap\widetilde{L}^1(0,T;\dot{B}^1_{
d,\infty})\big)^d\times L^\infty(0,T;\dot{B}^0_{d,\infty}).$$

We first state that $(\delta u, \delta\tau)\in \mathcal{H}_T$. From
the second equation of \eqref{cha} and Proposition \ref{Besovproperties}, Proposition \ref{BONY}, we have
\begin{equation*}
\begin{split}
&\|(\delta\tau)_t\|_{L^2_T(\dot{B}^0_{d,\infty})}
\lesssim\|(\delta\tau)_t\|_{L^2_T(\dot{B}^{\frac{d}{2}-1}_{2,1})}\\
&\lesssim\Big(\|u_1 \|_{L^2_T(\dot{B}^{\frac{d}{2}}_ {2,1})}+\|u_2
\|_{L^2_T(\dot{B}^{\frac{d}{2}}_ {2,1})}\Big)\Big(1
+\|\tau_1\|_{L^\infty_T(\dot{B}^\frac{d}{2}_{2,1})}+\|\tau_2\|_{L^\infty_T(\dot{B}^
\frac{d}{2}_{2,1})}\Big),
\end{split}
\end{equation*}
which means that $\delta\tau\in
C^{\frac{1}{2}}([0,T];\dot{B}^0_{d,\infty})$. A similar discussion
to $\delta \tau$ entails that $\delta u\in
C^{\frac{1}{2}}([0,T];\dot{B}^{-1}_{d,\infty})$.

Now, let us turn to the proof of estimates for $\delta\tau$. Applying Proposition
\ref{transport} to the second equation of \eqref{cha}, we have
\begin{equation*}
\begin{split}
\|\delta\tau\|_{\widetilde{L}^\infty_t(\dot{B}^0_{d,\infty})}
\lesssim& e^{C\int_0^t\|\nabla u_1\|_{\dot{B}^1_{d,1}}\mathrm{d}t'}\int_0^t e^{-C\int_0^{t'}\|\nabla u_1\|_{\dot{B}^1_{d,1}}\mathrm{d}s}\|F\|_{\dot{B}^0_{d,\infty}}\mathrm{d}t'
\\
\lesssim& e^{C\int_0^t\|\nabla
u_1\|_{\dot{B}^1_{d,1}}\mathrm{d}t'}\int_0^t\|F\|_{\dot{B}^0_{d,\infty}}\mathrm{d}t',
\end{split}
\end{equation*}
where $F=\mu_2D(\delta u)-\delta u
\cdot\nabla\tau_2-Q(\tau,\nabla\delta u)-Q(\delta\tau,\nabla u_2).$
By Proposition \ref{Besovproperties} and  Proposition \ref{BONY}, we get
\begin{equation*}
\begin{split}
\|F\|_{L^1_t(\dot{B}^0_{d,\infty})}\lesssim&\|u_2\|_{L^1_t(\dot{B}^{\frac{d}{2}+1}_{2,1})}
\|\delta\tau\|_{L^\infty_t(\dot{B}^0_{d,\infty})}\\
&+\big(1+\|\tau_1\|_{L^\infty_t(\dot{B}^{\frac{d}{2}}_{2,1})}+\|\tau_2\|_{L^\infty_t(\dot{B}^{\frac{d}{2}}_{2,1})}\big) \|\delta
u\|_{L^1_t(\dot{B}^1_{d,1})}.
\end{split}
\end{equation*}
We then finally obtain
\begin{equation}\label{deltatau}
\begin{split}
\|\delta\tau\|_{L^\infty_t(\dot{B}^0_{d,\infty}) }\lesssim&
e^{C\int_0^t\|\nabla
u_1\|_{\dot{B}^\frac{d}{2}_{2,1}}\mathrm{d}t'}\Big(\|u_2\|_{L^1_t(\dot{B}^{\frac{d}{2}+1}_{2,1})}
\|\delta\tau\|_{L^\infty_t(\dot{B}^0_{d,\infty})}
+\big(1+\|\tau_1\|_{L^\infty_t(\dot{B}^{\frac{d}{2}}_{2,1})}\\
&+\|\tau_2\|_{L^\infty_t(\dot{B}^{\frac{d}{2}}_{2,1})}\big) \|\delta
u\|_{L^1_t(\dot{B}^1_{d,1})}\Big).
\end{split}
\end{equation}
From Proposition \ref{log}, we infer
\begin{equation*}
\|\delta u\|_{L^1_t(\dot{B}^1_{d,1})}\lesssim \|\delta
u\|_{\widetilde{L}^1_t(\dot{B}^1_{d,\infty})} \log\bigg(
e+\frac{\|\delta u\|_{\widetilde{L}^1_t(\dot{B}^{0}_{d,\infty})}+
\|\delta u\|_{\widetilde{L}^1_t(\dot{B}^{2}_{d,\infty})}}{\|\delta
u\|_{\widetilde{L}^1_t (\dot{B}^1_{d,\infty})}}  \bigg).
\end{equation*}
Taking $T$ small enough such that
$\|u_2\|_{L^1_t(\dot{B}^{\frac{d}{2}+1}_{2,1})}$ sufficiently small,
then inserting the above inequality into \eqref{deltatau}, we deduce
\begin{equation}\label{tautotal}
\begin{split}
\|\delta\tau\|_{L^\infty_t(\dot{B}^0_{d,\infty}) }\lesssim C_T
\|\delta u\|_{\widetilde{L}^1_t(\dot{B}^1_{d,\infty})} \log\bigg(
e+\frac{\|\delta u\|_{\widetilde{L}^1_t(\dot{B}^{0}_{d,\infty})}+
\|\delta u\|_{\widetilde{L}^1_t(\dot{B}^{2}_{d,\infty})}}{\|\delta
u\|_{\widetilde{L}^1_t (\dot{B}^1_{d,\infty})}}  \bigg),
\end{split}
\end{equation}
where $C_T\triangleq
\exp{\big({C\|u_1\|_{L^1_T(\dot{B}^{\frac{d}{2}+1}_{2,1})}}\big)}
\big(1+\|\tau_1\|_{L^\infty_T(\dot{B}^{\frac{d}{2}}_{2,1})}
+\|\tau_2\|_{L^\infty_T(\dot{B}^{\frac{d}{2}}_{2,1})}\big)<\infty$.

For the estimate of $\delta u$. By  Proposition \ref{smoothing}, we
get
\begin{equation*}
\|\delta u\|_{L^\infty_t(\dot{B}^{-1}_{d,\infty})}+\|\delta u\|_{\widetilde{L}^1_t(\dot{B}^1_{d,\infty})}\lesssim \|E\|_{\widetilde{L}^1_t(\dot{B}^{-1}_{d,\infty})},
\end{equation*}
where $E=\mu_1\nabla\cdot \delta\tau -u_1\cdot\nabla \delta u-\delta
u\cdot\nabla u_2$. It follows from \eqref{fenbujifen1}, Remark \ref{hua} and Proposition \ref{BONY} that
\begin{equation*}
\begin{split}
\|E\|_{\widetilde{L}^1_t(\dot{B}^{-1}_{d,\infty})} &\lesssim\|\delta
\tau\|_{L^1_t(\dot{B}^0_{d,\infty})}+\|u_1\otimes\delta
u\|_{\widetilde{L}^1_t(\dot{B}^0_{d,\infty})}+\|\delta u\otimes
u_2\|_{\widetilde{L}^1_t(\dot{B}^0_{d,\infty})}\\
&\lesssim\|\delta \tau\|_{L^1_t(\dot{B}^0_{d,\infty})}+\big(\|u_1\|_{\widetilde{L}^2
_t(\dot{B}^1_{d,1})}+ \|u_2\|_{\widetilde{L}^2
_t(\dot{B}^1_{d,1})}\big)\|\delta u\|_{\widetilde{L}^2
_t(\dot{B}^0_{d,\infty})}.
\end{split}
\end{equation*}
Then in light of interpolation Theorem, we infer
\begin{equation}\label{utotal}
\begin{split}
&\|\delta u\|_{L^\infty_t(\dot{B}^{-1}_{d,\infty})}+\|\delta u\|_{\widetilde{L}^1_t(\dot{B}^1_{d,\infty})}\\
&\lesssim (\|u_1\|_{\widetilde{L}^2
_t(\dot{B}^1_{d,1})}+\|u_2\|_{\widetilde{L}^2
_t(\dot{B}^1_{d,1})})(\|\delta
u\|_{L^\infty_t(\dot{B}^{-1}_{d,\infty})}+\|\delta
u\|_{\widetilde{L}^1_t(\dot{B}^1_{d,\infty})})+\|\delta
\tau\|_{L^1_t(\dot{B}^0_{d,\infty})}.
\end{split}
\end{equation}
Recall that
$$u_t-\nu\Delta u=\mu_1\mathbb{P}\nabla\cdot\tau-\mathbb{P}(u\cdot\nabla u).$$
It is easy to verify that
$\mu_1\mathbb{P}\nabla\cdot\tau-\mathbb{P}(u\cdot\nabla u)\in
L^1_T(\dot{B}^0_{d,1})$ for the finite time $T$. Then from
Proposition \ref{smoothing}, we infer that
$u_i\in\widetilde{L}^{\infty}_T(\dot{B}^0_{d,1})\cap\widetilde{L}^{1}_T(\dot{B}^2_{d,1})$
($i=1,2$), therefore $u_i\in\widetilde{L}^2_T(\dot{B}^1_{d,1})$ by
interpolation theorem. Taking $T$ small enough such that
$\|u_i\|_{\widetilde{L}^2_t{\dot{B}^{1}_{d,1}}}$ sufficiently small, then combining
\eqref{tautotal}, \eqref{utotal}, we have
\begin{equation}\label{utautotal}
\begin{split}
\|\delta u\|_{L^\infty_t(\dot{B}^{-1}_{d,\infty})}+&\|\delta u\|_{\widetilde{L}^1_t(\dot{B}^1_{d,\infty})}\\
&\lesssim C_T \int_0^t \|\delta
u\|_{\widetilde{L}^1_{t'}(\dot{B}^1_{d,\infty})} \log\bigg(
e+\frac{\|\delta u\|_{\widetilde{L}^1_{t'}(\dot{B}^{0}_{d,\infty})}+
\|\delta
u\|_{\widetilde{L}^1_{t'}(\dot{B}^{2}_{d,\infty})}}{\|\delta
u\|_{\widetilde{L}^1_{t'} (\dot{B}^1_{d,\infty})}}
\bigg)\mathrm{d}t'.
\end{split}
\end{equation}
Denote $X(t)\triangleq\|\delta u\|_{L^\infty_t(\dot{B}^{-1}_{d,\infty})}+\|\delta u\|_{\widetilde{L}^1_t(\dot{B}^1_{d,\infty})}$, $V(t)\triangleq\|\delta u\|_{\widetilde{L}^1_{t}(\dot{B}^{0}_{d,\infty})}+
\|\delta u\|_{\widetilde{L}^1_{t}(\dot{B}^{2}_{d,\infty})}$, we claim that $V(T)<\infty$. In fact
\begin{equation*}
\begin{split}
\|\delta u\|_{\widetilde{L}^1_{T}(\dot{B}^{0}_{d,\infty})}+
\|\delta u\|_{\widetilde{L}^1_{T}(\dot{B}^{2}_{d,\infty})}&\lesssim \sum_{i=1}^2\int_0^T(\|u_i\|_{\dot{B}^0_{d,1}}+\|u_i\|_{\dot{B}^2_{d,1}})\mathrm{d}t'\\
&\lesssim\sum_{i=1}^2(T\|u_i\|_{L^\infty_T{\dot{B}^{\frac{d}{2}-1}_{2,1}}}
+\|u_i\|_{L^1_T{\dot{B}^{\frac{d}{2}+1}_{2,1}}}).
\end{split}
\end{equation*}
We can rewrite \eqref{utautotal} as follows
\begin{equation}\label{utautotal1}
\begin{split}
X(t)\lesssim C_T\int_0^t X(t')\log\left( e+\frac{V(T)}{X(t')}\right).
\end{split}
\end{equation}
As
$$\int_0^1\frac{\mathrm{d}r}{r\log(e+\frac{V(T)}{r})}=+\infty,$$
Osgood lemma implies $X\equiv0$ on $[0,T]$, whence also
$\delta\tau\equiv0$. Then a continuity argument ensures that $(u_1,\tau_1)=(u_2, \tau_2)$ on $\mathbb{R}^+$.

\section{Appendix}\label{Appendix}

This section is devoted to the estimates of the convection terms
which were used in Section 3.

First, let us give some definitions in paradifferential calculus in
homogeneous spaces. We designate the homogeneous paraproduct of $v$
by $u$ as
$$\dot{T}_uv\triangleq\sum_q\dot{S}_{q-1}\dot{\Delta}_qv.$$
and the homogeneous remainder of $u$ and $v$ as
$$\dot{R}(u,v)\triangleq\sum_{q}\dot{\Delta}_qu\dot{\widetilde{\Delta}}_qv,\ \ \text{and}\ \ \dot{\widetilde{\Delta}}_q=\dot{\Delta}_{q-1}+\dot{\Delta}_{q}+\dot{\Delta}_{q+1}.$$
Formally, we have the following homogeneous Bony decomposition:
$$uv=\dot{T}_uv+\dot{T}_vu+\dot{R}(u,v).$$
The properties of continuity of homogeneous paraproduct and remainder on homogeneous hybrid Besov spaces are described as follows.
\begin{proposition}\label{pro1}
For all $s_1,s_2,t_1,t_2$ such that $s_1\leq\frac{d}{2}$ and $s_2\leq\frac{d}{2}$, the following estimate holds
$$\|T_uv\|_{\dot{B}^{s_1+t_1-\frac{d}{2},s_2+t_2-\frac{d}{2}}}\lesssim\|u\|_{\dot{B}^{s_1,s_2}}\|v\|_{\dot{B}^{t_1,t_2}}.$$
If $\min(s_1+t_1,s_2+t_2)>0$, then
$$\|R(u,v)\|_{\dot{B}^{s_1+t_1-\frac{d}{2},s_2+t_2-\frac{d}{2}}}\lesssim\|u\|_{\dot{B}^{s_1,s_2}}\|v\|_{\dot{B}^{t_1,t_2}}.$$
If $u\in L^{\infty}$,
$$\|T_uv\|_{\dot{B}^{t_1,t_2}}\lesssim\|u\|_{L^{\infty}}\|v\|_{\dot{B}^{t_1,t_2}},$$
and, if $\min(t_1,t_2)>0$, then
$$\|R(u,v)\|_{\dot{B}^{t_1,t_2}}\lesssim\|u\|_{L^{\infty}}\|v\|_{\dot{B}^{t_1,t_2}}.$$
\end{proposition}
\begin{remark}\label{BY}
When $d\geq2$, we have $\|uv\|_{\dot{B}^{\frac{d}{2}-1,\frac{d}{2}}}\lesssim
\|u\|_{\dot{B}^{\frac{d}{2}}_{2,1}}\|v\|_{\dot{B}^{\frac{d}{2}-1,\frac{d}{2}}}$.
\end{remark}

\begin{proposition}\label{pro2}
Let $u$ be a vector with $\nabla\cdot u=0$. Suppose that $-1-\frac{d}{2}<s_1,t_1,s_2,t_2\leq1+\frac{d}{2}$. The following two estimates hold
\begin{eqnarray*}
&&\big|(\dot{\Delta}_j(u\cdot\nabla v),\dot{\Delta}_jv)\big|
\lesssim
c_j2^{-j\psi^{s_1,s_2}(j)}\|u\|_{\dot{B}_{2,1}^{\frac{d}{2}+1}}\|v\|_{\dot{
B}^{s_1,s_2}}\|\dot{\Delta}_jv\|_{L^2},\\
&&
\begin{split}\big|(\dot{\Delta}_j(u\cdot\nabla v),\dot{\Delta}_jw)+&(\dot{\Delta}_j(u\cdot\nabla w),\dot{\Delta}_jv)\big|
\lesssim
c_j\|u\|_{\dot{B}_{2,1}^{\frac{d}{2}+1}}(2^{-j\psi^{s_1,s_2}(j)}\|v\|_
{\dot{B}^{s_1,s_2}}\|\dot{\Delta}_jw\|_{L^2}\\
&+2^{-j\psi^{t_1,t_2}(j)}
\|w\|_{\dot{B}^{t_1,t_2}}\|\dot{\Delta}_jv\|_{L^2}).\end{split}
\end{eqnarray*}
where the function $\psi^{\alpha,\beta}(j)$ define as
$\psi^{\alpha,\beta}(j)=\alpha$ if $j\leq 0$,
$\psi^{\alpha,\beta}(j)=\beta$, if $j>0$, and
$\sum_{j\in\mathbb{Z}}c_j\leq1$.
\end{proposition}
One can refer to \cite{R.DANCHIN2001} for the proof of above two Propositions. Here we only have made a slight modification since the incompressible condition on $u$.

Next, we introduce a useful Proposition to deal with $[\mathbb{P}\mathrm{div},u\cdot\nabla]$ type commutators.
\begin{proposition}\label{Profile}
For any smooth tensor $[\tau^{i,j}]_{d\times d}$ and $d$ dimensional vector $u$, it always holds that
$$\mathbb{P}\nabla\cdot(u\cdot\nabla\tau)=\mathbb{P}(u\cdot\nabla\mathbb{P}\nabla\cdot\tau)+\mathbb{P}(\nabla u\cdot\nabla\tau)-\mathbb{P}(\nabla
u\cdot\nabla\Delta^{-1}\nabla\cdot\nabla\cdot\tau),$$
where the ith component of $\nabla u\cdot\nabla\tau$ is
$$[\nabla u\cdot\nabla \tau]^i=\sum_{j}\partial_j u\cdot\nabla\tau^{i,j},$$
and also
$$[\nabla u\cdot\nabla\Delta^{-1}\nabla\cdot\nabla\cdot\tau]^i
=\partial_iu\cdot\nabla\Delta^{-1}\nabla\cdot\nabla\cdot\tau.$$
\end{proposition}
For more detailed derivations, one can refer to the proof of the three dimensions  in \cite{zhuyi}.

\vspace{.5cm}

\noindent\textbf{Acknowledgements.} Q. Chen and X. Hao  were
supported by the National Natural Science Foundation of China
(No.11671045).

\end{document}